\documentclass[11pt, reqno]{amsart}

\usepackage{amsmath, amssymb, amsthm}
\usepackage{mathtools}
\usepackage{esint}

\usepackage[colorlinks=true, urlcolor=blue, linkcolor=blue, citecolor=blue, pdfstartview=]{hyperref}

\usepackage{color}
\newtheorem*{theorem*}{Theorem}
\newtheorem{theorem}{Theorem}
\newtheorem{lemma}{Lemma}
\newtheorem{proposition}{Proposition}

\newtheorem{remark}{Remark}

\numberwithin{equation}{section}

\begin{document}

\title [Quantitative comparison theorems]{Quantitative comparison theorems in Riemannian and K\"ahler geometry}
\author{Kwok-Kun Kwong}
\address{Department of Mathematics, National Cheng Kung University, Tainan City 70101, Taiwan}
\address{Current address: School of Mathematics and Statistics, University of Sydney, NSW 2006, Australia}
\email{kwok-kun.kwong@sydney.edu.au}
%\email{kwong@mail.ncku.edu.tw}


\begin{abstract}
We obtain sharp quantitative Laplacian upper and lower estimates under no assumption on curvatures. As a result, we derive quantitative Laplacian, area and volume comparison theorems for tubes in Riemannian and K\"ahler manifolds under weak integral curvature assumptions. We also give some applications, such as a general Bonnet-Myers theorem and Cheng's eigenvalue estimate under weak integral curvature assumptions.
\end{abstract}
\maketitle

\section{Introduction}
The aim of this paper is to look for weaker conditions than a lower Ricci curvature bound or upper sectional curvature bound such that some classical comparison theorems, such as the Laplacian comparison, Bishop-Gromov volume comparison theorem and G\"unther's theorem still hold.

The motivation is as follows. For the standard proof of the Bishop-Gromov volume comparison theorem, only the Ricci curvature lower bound in the radial direction is used. And most often a comparison theorem holds true already for the area or volume element after some Sturm-Liouville type ODE argument. Since the area of the geodesic sphere is just the integral of the area element it is natural to expect that some kind of lower bound for the ``integral'' of the Ricci curvature in the radial directions should be enough to guarantee an area comparison theorem for the geodesic sphere.
On the other hand, the naive approach of replacing the Ricci curvature simply by the average of Ricci curvatures over all directions, i.e. the scalar curvature, does not work because there are counter-examples (Remark \ref{rem: sc}). It turns out that a weighted version of the integral of the Ricci curvature suffices to ensure Laplacian and volume comparison (Theorems \ref{thm: est}, \ref{thm: area vol est}). This direction of research has been previously pursued in a number of papers, such as \cite{avez1972riemannian}, \cite{calabi1967ricci}, \cite{gallot1988isoperimetric}, \cite{markvorsen1982ricci}, and relatively more recently \cite{peterson1997comparison}, \cite{petersen1998integral}, \cite{petersen1997relative}. See also \cite{li2005comparison}, \cite{tam2012some} and the recent paper \cite{ni2018comparison} for comparison results under various pointwise but weaker types of curvature bounds.

In the first part of this paper, for each function $k(t)$, we obtain a corresponding Laplacian estimate, an area estimate for geodesic spheres, as well as a volume estimate for geodesic balls with no condition on the curvature on $M$ (in particular, no assumption on the lower bound of the Ricci curvature), as long as we stay within the cut locus or the injectivity radius. These estimates lead to comparison for Laplacian, area or volume under some weak integral curvature assumptions.
In the Riemannian case, the function $ k(t)$ can be thought of as the Ricci curvature of the model warped product space in the radial direction. The flexibility of choosing $ k(t)$ allows us to obtain for example a fairly general version of Bonnet-Myers theorem (Theorem \ref{thm: bonnet myers}) and its K\"ahler analogue (Theorem \ref{thm: cpx myers}). The main argument relies on the second variational formula with a critical use of the index lemma. Compared to the approach of using ODE analysis, this approach is often more ``linear'' as it avoids estimating the solution of some nonlinear Riccati type differential inequality (cf. \cite{eschenburg1990comparison}).

For example, the first result we will prove is the following Laplacian estimate:
\begin{theorem*}[Theorem \ref{thm: est}]
Let $x=(r, \theta)$ in geodesic polar coordinates centered at $p$.
If $ s_k> 0$ on $(0,r]$, then
\begin{align*}\label{ineq: lap d1}
\Delta r(x)
\le (n-1)\frac{ s_k'(r)}{ s_k(r)}-\int_{0}^{r}\widehat{\mathrm{Ric}}_k \left(\frac{s_k(t)}{s_k(r)}\gamma_\theta'(t) \right) dt.
\end{align*}
If $s_k>0$ on $(0, \sup d_p)$, this also holds in the sense of distribution if the second term on R.H.S. is interpreted properly.
\end{theorem*}
Here $\gamma_\theta$ is the geodesic with initial vector $\theta$, $s_k''(t)+k(t)s_k(t)=0$ with $s_k(0)=0$, $s_k'(0)=1$ for a continuous but otherwise arbitrary function $k(t)$, and $\widehat {\mathrm{Ric}}_k(v)=\mathrm{Ric}(v, v)-(n-1)k(t)g(v,v)$ for $v\in T_{\gamma_\theta(t)}M$.
We remark that it is natural to allow $k$ to be non-constant when we consider the existence of conjugate points, see Proposition \ref{prop: conj}.

The above estimate is the starting point of the later results, such as the quantitative relative area/volume comparison theorems, Bonnet-Myers theorem and Cheng's eigenvalue estimate.
Note that there is no curvature assumption. Indeed, a feature of our results is that there is an explicit appearance of the curvature term in our estimates (Theorems \ref{thm: est}, \ref{thm: area vol est}, \ref{thm: submfd}, \ref{thm: kahler}, \ref{thm: cpx submfd}, \ref{thm: gunther1}, \ref{thm: gunther area vol est}, \ref{thm: gunther submfd}, \ref{thm: cpx gunther1}, \ref{thm: cpx gunther area vol est}, \ref{thm: cpx gunther submfd}), which is independent of any curvature condition. On the other hand, our results also give sharper estimates. E.g. we prove the monotonicity for the volume ratio of geodesic balls:
\begin{theorem*}[Theorem \ref{thm: area vol est}]
If $s_k>0$ on $ (0, r]$, then
\begin{align*}
\frac{d}{dr}\left(\frac{|B_g(r, p)|} {|B_{\overline g}(r)|} \right)
\le -\frac{s_k(r)^{n-1}}{|B_{\overline g}(r)|^2} \int_{0}^{r} \frac{|B_{\overline g}(u)|}{s_k(u)^{n+1}} \int_{B_g(u, p)}\widehat{\mathrm{Ric}}_k ({s_k(t)}\partial_t ) dV\, du.
\end{align*}
The equality holds if and only if $B_g(r, p)$ is isometric to $ B_{\overline g}(r)$, where $\overline g=dt^2+s_k(t)^2 g_{\mathbb S^{n-1}}$.

In particular, if
$\displaystyle \int_{B_g(u, p)} \widehat{\mathrm{Ric}}_k \left(s_k(t) \partial_t \right) dV\ge 0$
for all $u\in (0, r)$, then $\frac{| B_g(u, p)|}{| B_{\overline g}(u)|}$ is non-increasing on $(0, r)$.
\end{theorem*}

This can be regarded as a quantitative version of the Bishop-Gromov volume comparison theorem. This result makes the defect of the volume of the geodesic ball to that of the standard one easier to measure. Our approach to area and volume estimates is also quite robust in the sense that it can be easily adapted to estimate other similar area or volume-type integrals, such as those with weight.

The rest of this paper is as follows. In Section \ref{sec: riem}, we derive quantitative versions of various comparison theorems on a Riemannian manifold, first for distance from a fixed point and then a submanifold. A general Bonnet-Myers' theorem will also be derived. Then in Section \ref{sec: kahler}, we extend the results to K\"ahler manifolds. This is not a direct application of the results in Section \ref{sec: riem} because of the special geometry of K\"ahler manifolds. For our purpose we introduce the notion of $\ell$-holomorphic sectional curvatures, which is the K\"ahler version of the $\ell$-sectional curvatures defined in \cite{li2005comparison}. In Section \ref{sec: gunther}, we prove quantitative and relative G\"unther-type results in Riemannian and K\"ahler manifolds, i.e. lower bound for the volume of tubes around a submanifold. In the course, we will see that there is a curvature quantity, expressed explicitly in Fermi coordinates, which is analogous to the role of Ricci curvature in Bishop-Gromov comparison theorem, see Theorem \ref{thm: gunther1}. Finally, we give some applications in Section \ref{sec: applications}, such as Cheng's eigenvalue estimate (Theorem \ref{thm: Cheng first thm}) assuming only a lower bound on the integral of some curvature quantities on subsets of metric balls.

In the future, we plan to investigate the Lorentzian analogue of these results under some weak energy conditions, which is of physical interest to understand singularity theorems in general relativity. It also seems plausible that these results can be extended to quaternionic K\"ahler manifolds, which we do not do here for simplicity.

{\sc Acknowledgments}:
We would like to thank Hojoo Lee, Man-Chun Lee, Miles Simon and Ye-Kai Wang for useful discussions. The research of the author is partially supported by Ministry of Science and Technology in Taiwan under grant MOST 106-2115-M-006-017-MY2.

\section{Comparison results in Riemannian manifolds}\label{sec: riem}
\subsection{Notions and preliminaries}
Let us explain our notation. Throughout this paper, all manifolds and submanifolds are assumed to be complete, connected and orientable unless specified otherwise. Let $(M, g)$ denotes an $ n$-dimensional Riemannian manifold. Let $ k(t)$
be a continuous function on an interval $I$ containing $ 0$
and $s_k(t)$ be the solution to the equation
\begin{equation}\label{eq: s}
s_k''(t)=-k(t) s_k(t), \quad s_k(0)=0, \quad s_k'(0)=1.
\end{equation}

Often, we will compare $(M, g)$ with the ``model space'' defined as the warped product manifold $ (\overline M=[0, r_0)\times \mathbb S^{n-1}, \overline g)$, where $ \overline g=dt^2+s_k(t)^2 g_{\mathbb S^{n-1}}$ and $ g_{\mathbb S^{n-1}}$ is the standard round metric on the unit sphere $ \mathbb S^{n-1}$. Of course, for $ (\overline M, \overline g)$ to be truly a Riemannian manifold we at least require $ s_k>0$ on $ (0, r_0)$. However, to be flexible we do not want to impose any condition now and the assumption on $s_k$ will be stated in the results.

It is easy to see that the Ricci curvature of $\overline g$
in the radial direction is $\overline {\mathrm{Ric}}(\partial_t, \partial_t)=-(n-1)\frac{s_k''(t)}{s_k(t)}=(n-1)k(t)$.
When $k$ is constant then
\begin{align*}\label{eq: sk}
s_k(t)=
\begin{cases}
\frac{1}{\sqrt{k}} \sin \left(\sqrt{k}t\right)\quad &\textrm{ if }k>0\\
t\quad &\textrm{ if }k=0\\
\frac{1}{\sqrt{-k}}\sinh \left(\sqrt{-k}t\right)\quad &\textrm{ if }k<0.
\end{cases}
\end{align*}

Fix $p\in M$ for the moment.
Let $\gamma_\theta(t)$ be the geodesic starting from $ p$ with initial vector $\theta\in S_pM =\{\theta\in T_pM : |\theta|=1\}$. We define
\begin{equation*}
\widehat{\mathrm{Ric}}_k(u):= \mathrm{Ric}(u, u)-(n-1)k(t) g_{\gamma_\theta(t)}(u, u)
\end{equation*}
for $u\in T_{\gamma_\theta(t)}M$.

Of course, the function $\widehat{\mathrm{Ric}}_k$ depends on the direction $ u\in T_xM$. If $k$ is constant and a scalar function is preferred, we can define $ \displaystyle \widehat {\underline{\mathrm{Ric}}}_k(x):=\min_{u\in S_xM}\widehat{\mathrm{Ric}}_k(u)$. Then clearly, $ \widehat{\mathrm{Ric}}_k(u)\ge \widehat{\underline{\mathrm{Ric}}}_k(x) $ for all $ u\in S_xM$. For  results in Section \ref{sec: riem}, we can replace conditions involving $ \widehat{\mathrm{Ric}}_k$ by  $ \widehat{\underline{\mathrm{Ric}}}_k\left(\gamma(t)\right)$. Comparison theorems involving (the negative part of) $\widehat{\underline{\mathrm{Ric}}}_k$ are given in \cite{gallot1988isoperimetric}, \cite{petersen1997relative}, \cite{peterson1997comparison}, which use a different approach to obtain estimates. Our approach is more direct, and take into account both the positive and negative part of $\widehat {\mathrm{Ric}}_k$.

For later use, we also need the notion of some ``average'' curvature with strength lying between the sectional curvature and the Ricci curvature.
Let $K(w, v)=\frac{\langle R(w, v)v, w\rangle}{|w\wedge v|^2}$ be the sectional curvature of the plane spanned by $ w$ and $ v$. As in \cite{li2005comparison}, the $ \ell$-sectional curvature is defined by $ K^\ell(W, v):=\sum_{i=1}^\ell K(e_i, v)$ where $ v$ is non-zero, $ W$ is an $ \ell$-dimensional subspace orthogonal to $ v$ and $ e_i$ is an orthonormal basis of $ W$. By convention, $K^0=0$. Along a given geodesic $\gamma$, define also
\begin{equation*}\label{eq: Kl}
\widehat {K^\ell}_k(W, v):=K^\ell(W, v)-\ell k(t)g(v,v)
\end{equation*}
for $v\in T_{\gamma(t)}M$ and $W$ an $\ell$-dimensional subspace of $T_{\gamma(t)}M$ orthogonal to $v$.

In particular, if $\ell=n-1$, they reduce to the Ricci curvature and $\widehat{\mathrm{Ric}}_k$ respectively. For comparison theorem of the Laplacian of the distance from a point or a hypersurface, the notion of $ \widehat{\mathrm{Ric}}_k$ is enough. But for comparing distance function from an $ \ell$-dimensional submanifold, the curvature $ \widehat {K^\ell}_k$ and $\widehat K^{n-1-\ell}_k$ are involved. Again, if a scalar function is needed, then we can always replace $ \widehat K^\ell_k$ by $ \widehat {\underline{K}^\ell}_k(x):=\min \widehat {K^\ell}_k(W, v)$ where the minimum is taken over the set $ \{(W, v): v\in S_xM, W<v^\perp \textrm{ is $\ell $-dimensional}\}$.

Recall that a Jacobi field $Y(t)$ along a normal geodesic $ \gamma$ orthogonal to $ \Sigma$ is said to be adapted to a submanifold $ \Sigma$ if $ Y(0)\in T\Sigma$ and $ Y'(0)-A_{\gamma'(0)} Y(0)\in N\Sigma$. If $ \Sigma$ is a point, the initial condition is just $ Y(0)=0$.
A standard but useful fact is the following (cf. e.g. \cite[p. 3]{schoen1994lectures}).
\begin{proposition}\label{prop: hess}
Let $\Sigma$ be an embedded submanifold of $ M$ and let $ r=d_\Sigma: M\to \mathbb R$ be the distance from $ \Sigma$ on $ M$.
If $x$ lies within the cut locus of $ \Sigma$ and $ X\in \nabla r(x) ^\perp$, then
\begin{align*}
\nabla ^2 r(X, X)=&\int_{0}^{r} \left(|Y'(t)|^2-\langle R(Y, \partial_r)\partial_r, Y\rangle\right) dt +A_{\partial_r}(Y(0), Y(0))\\
=&:I_\Sigma(Y, Y)
\end{align*}
where $Y$ is the $ \Sigma$-adapted Jacobi field along the minimizing normal geodesic $ \gamma$ emanating from $ \Sigma$, satisfying $ Y(r)=X$. Here $ A$ is the second fundamental form defined by $ A_v(X, Y)=\langle v, -\nabla_XY\rangle $ and $ I_\Sigma$ is the index form with respect to $ \Sigma$. (If $ \Sigma$ is a point, then $ A=0$.)
\end{proposition}

We end this subsection with an extension of some familiar facts in trigonometry.
Let $c_k(t)$ be the solution to
\begin{align}\label{eq: c}
c_k''(t)=-k(t)c_k(t), \quad c_k(0)=1, \quad c_k'(0)=0.
\end{align}
Again, when $k$ is contant, then
\begin{align*}\label{eq: ck}
c_k(t)=
\begin{cases}
\frac{1}{\sqrt{k}} \cos \left(\sqrt{k}t\right)\quad &\textrm{ if }k>0\\
1\quad &\textrm{ if }k=0\\
\frac{1}{\sqrt{-k}}\cosh \left(\sqrt{-k}t\right)\quad &\textrm{ if }k<0.
\end{cases}
\end{align*}
To further simplify notation, we define $\mathrm{ct}_k (r)=\frac{c_k(r)}{s_k(r)}$ and $\mathrm{tg}_k (r)=\frac{s_k(r)}{c_k(r)}$.
\begin{lemma}
\label{lem: aux}
\noindent
\begin{enumerate}
\item\label{item1}
$s_k(t)c_k'(t)-c_k(t)s_k'(t)=-1$.
\item
$\mathrm{ct}_k'(t)=-\frac{1}{s_k(t)^2}$. In particular, $\mathrm{ct}_k$ is decreasing on any interval contained in its domain.
\item
$\mathrm{tg}_k'(t)=\frac{1}{c_k(t)^2}$. In particular, $\mathrm{tg}_k$ is increasing on any interval contained in its domain.
\end{enumerate}
\end{lemma}
\begin{proof}
By definition,
$(s_k c_k'-c_ks_k')'=s_kc_k''-c_ks_k''=-ks_kc_k+ks_kc_k=0$ and so $ s_kc_k'-c_ks_k'$ is a constant, which is equal to $-1$ by the initial conditions. The rest follows from this.
\end{proof}

\subsection{Comparison theorems for distance from a point}\label{subsect: Riem}

Using geodesic polar coordinates centered at a fixed point $p\in M$, within the cut locus of $p$, the volume element on $M$ can be expressed as $ dV=F(r, \theta)dr d\theta$, where $\theta\in S_pM\cong \mathbb S^{n-1}$ and $ d\theta$ is the volume element of $ \mathbb S^{n-1}$. In this subsection ($ \overline F$ will change in later subsections), we let
\begin{equation*}
\overline F(r)=s_k(r)^{n-1},
\end{equation*}
which is the corresponding volume density of $(\overline M=[0, r_0)\times \mathbb S^{n-1}
, \overline g=dt^2+s_k(t)^2 g_{\mathbb S^{n-1}})$ in polar coordinates.
Let $d_p=d(p, \cdot)$ be the distance function on $ M$. We use $ '$ to denote partial derivative with respect to the radial direction $ r$. E.g. $ F'(r, \theta)=\frac{\partial F}{\partial r}(r, \theta)$.

Strictly speaking, Theorem \ref{thm: est} below is a special case of Theorem \ref{thm: submfd}. However, we choose to present it here not only because of simpler presentation, but also because it will be useful later, and is indeed one of the steps in the proof of Theorem \ref{thm: submfd}.
\begin{theorem} \label{thm: est}
\noindent
\begin{enumerate}
\item\label{est1}
Let $x=(r, \theta)$ in geodesic polar coordinates.
Assume there is no cut point of $p$ along $ \gamma_\theta$ on $ [0, r]$. If $ s_k(r)\ne 0$, then
\begin{align}\label{ineq: lap d1}
\Delta d_p(x)
=\frac{F'(r, \theta)}{F(r, \theta)}
\le \frac{\overline F'(r)}{\overline F(r)}-\int_{0}^{r}\widehat{\mathrm{Ric}}_k \left(\frac{s_k(t)}{s_k(r)}\gamma_\theta'(t) \right) dt.
\end{align}
\item
Assume $s_k>0$ on $ (0, \sup_M d_p)$,
then \eqref{ineq: lap d1} also holds in the sense of distribution. i.e. for $0\le f\in C^\infty_c(M)$ and $ r=d_p$,
\begin{align*}
\int_M r \Delta f\le \int_M f
\left[\frac{\overline F'(r)}{\overline F(r)}-\int_{0}^{r} \widehat{\mathrm{Ric}}_k\left(\frac{s_k(t)}{s_k(r)}\partial_t\right)dt\right].
\end{align*}
(Note that $\int_{0}^{r}\widehat{\mathrm{Ric}}_k\left(\frac{s_k(t)}{s_k(r)}\partial_t\right)dt$ is well-defined almost everywhere as a function on $ M$.)

\item\label{est2}
Assume there is no cut point of $p$ along $ \gamma_\theta$ on $ [0, r]$.
If $s_k>0$ on $ (0, r]$, then
\begin{equation*}
\label{ineq: F1}
\begin{split}
F(r, \theta)\le&
\exp \left[-\int_{0}^{r}\int_{0}^{\rho} \widehat{\mathrm{Ric}}_k \left(\frac{s_k(t)}{s_k(\rho)}\gamma_\theta'(t) \right) dt \, d\rho \right] \overline F(r)\\
=& \exp \left[-\int_{0}^{r} \left(\mathrm{ct}_k(t)-\mathrm{ct}_k(r)\right) \widehat{\mathrm{Ric}}_k\left(s_k(t)\gamma_\theta'(t)\right)dt
\right] \overline F(r).
\end{split}
\end{equation*}
(Note that $\mathrm{ct}_k(t)-\mathrm{ct}_k(r)> 0$ by Lemma \ref{lem: aux}.)

If $k(t)=k$ is constant, then this inequality can also be expressed as
\begin{equation*}\label{eq: k const F}
F(r, \theta)\le\exp \left[-\frac{1}{s_k(r)}\int_{0}^{r} s_k(t)s_k(r-t) \widehat{\mathrm{Ric}}_k(\gamma_\theta'(t))dt\right] \overline F(r).
\end{equation*}

\end{enumerate}
\end{theorem}
\begin{proof}
\begin{enumerate}
\item
Let $e_1, e_2, \cdots, e_n= \theta$ be a positively oriented orthonormal basis of $ T_pM$ and $ E_i$ be the parallel translation of $ e_i$ along $ \gamma_\theta$. Define $ \{Y_i^{r, \theta}(t)\}_{i=1}^{n-1}$ to be the unique Jacobi fields along $ \gamma_\theta$ with $ Y_i^{r, \theta}(0)=0$ and $ Y_i^{r, \theta}(r)=E_i(r)$. For convenience we simply denote $ Y_i^{r, \theta}$ by $ Y_i$ and $ \gamma_\theta$ as $ \gamma$.

As $\langle Y_i(t), \gamma'(t)\rangle=0 $ at $ t=0$ and $ t=r$, we have $ \langle Y_i(t), \gamma'(t)\rangle \equiv 0$, i.e. $ Y_i(t)$ are tangential to $ S_t$.
Since $Y_i(t)=d\exp_p|_{t\theta}(t{Y_i}\, '(0))$, we see that the $(n-1)$-dimensional Jacobian satisfies
$F(t, \theta) =\frac{\det\left(Y_1(t), \cdots, Y_{n-1}(t)\right)}{\det\left({Y_1} '(0), \cdots, {Y_{n-1}} '(0)\right)} $.
Note that $F(t, \theta)$ depends on $ (t, \theta)$ only and is independent of $ r$ and $ Y_i$.
We have the formula (cf. \cite[p. 460]{heintze1978general})
\begin{equation}\label{eq: log F}
\begin{split}
\left(\log F \right)'(r, \theta)
=\left[\log \left(\det (Y_1, \cdots, Y_{n-1})\right)\right]'(r)
=&\sum_{i=1}^{n-1}\int_{0}^{r}\left(\langle {Y_i}', {Y_i}'\rangle -\langle R (Y_i, \gamma') \gamma', Y_i\rangle\right) dt\\
=&\sum_{i=1}^{n-1} I(Y_i, Y_i).
\end{split}
\end{equation}
Let
$X_i(t)=\frac{s_k(t)}{s_k(r)}E_i(t)$.
Then by the index lemma \cite[Ch. III, Lemma 2.10]{sakai1996riemannian}
\begin{equation}\label{ineq: index}
I(Y_i, Y_i) \le I(X_i, X_i).
\end{equation}
By integration by parts,
\begin{equation}\label{eq: IX}
\begin{split}
I(X_i, X_i)
=&\int_{0}^{r}\left(-\langle {X_i}'', X_i\rangle -\langle R (X_i, \gamma') \gamma', X_i\rangle \right)dt+\langle X_i(r), {X_i}'(r)\rangle\\
=&-\int_{0}^{r}\frac{s_k(t)^2}{s_k(r)^2} \widehat {K^1}_k(E_i, \gamma')dt+\frac{s_k'(r)}{s_k(r)}.
\end{split}
\end{equation}
Summing \eqref{eq: IX} on $i=1, \cdots, {n-1}$ and
combining with \eqref{eq: log F}, \eqref{ineq: index}, we have
\begin{align*}\label{ineq: lap d1}
\left(\log F\right)'(r, \theta)
\le&-\int_{0}^{r} \widehat{\mathrm{Ric}}_k \left(\frac{s_k(t)}{s_k(r)}\gamma_\theta'(t) \right) dt+ (\log \overline F'(r).
\end{align*}
Observe that $F=\sqrt{\det (g_{ij})}$ in polar coordinates, and using $ \Delta f=\frac{1}{\sqrt{\det (g_{ij})}}\partial_i (\sqrt{\det (g_{ij})}g^{ij}\partial_jf)$, we see
that $\Delta d_p(x)= \frac{F'(r, \theta)}{F(r, \theta)}$. So \eqref{est1} follows.
\item
Let $\phi(r, \theta)=\int_{0}^{r}\widehat{\mathrm{Ric}}_k\left(\frac{s_k(t)}{s_k(r)}\gamma_\theta'(t)\right)dt$, $\overline H(r)=\frac{\overline F'(r)}{\overline F(r)}$
and $\mathrm{c}(\theta)$ be the cut distance in the direction $\theta$. Then for $ r<\mathrm{c}(\theta)$, as $ \overline H(r)-\phi(r, \theta)-\Delta d_p \ge 0$,
\begin{equation*}\label{ineq: H-R}
(\overline H(r)-\phi(r, \theta)F(r, \theta)\ge F(r, \theta)\Delta d_p=F'(r, \theta).
\end{equation*}
Multiply this inequality by a non-negative $f\in C_c^\infty(M)$ and proceed in the same way as \cite[Theorem 4.1]{li1993lecture}, we can prove \eqref{ineq: lap d1} in the distributional sense. We omit the details.
\item
Integrating \eqref{ineq: lap d1} gives (note that $\log F(r, \theta)-\log \overline F(r)\to 0$ as $ r\to 0^+$)
\begin{align*}
\log F(r, \theta)
\le&-\int_{0}^{r}\int_{0}^{\rho} \widehat{\mathrm{Ric}}_k \left(\frac{s_k(t)}{s_k(\rho)}\gamma_\theta'(t) \right) dt\, d\rho + \log \overline F(r).
\end{align*}
\begin{equation*}\label{ineq: F}
\textrm{i.e. }\quad F(r, \theta)
\le\exp \left[-\int_{0}^{r}\int_{0}^{\rho} \widehat{\mathrm{Ric}}_k \left(\frac{s_k(t)}{s_k(\rho)}\gamma_\theta'(t) \right) dt \, d\rho \right] \overline F(r).
\end{equation*}
We can transform the double integral into a single integral using the function $c$.
By Fubini's theorem and Lemma \ref{lem: aux},
\begin{equation}
\label{eq: sing int}
\begin{split}
\int_{0}^{r}\int_{0}^{\rho} \widehat{\mathrm{Ric}}_k \left(\frac{s_k(t)}{s_k(\rho)}\gamma_\theta'(t) \right) dt \, d\rho
=&\int_{0}^{r}\left(\int_{t}^{r}\frac{1}{s_k(\rho)^2}d\rho\right)\widehat{\mathrm{Ric}}_k(s_k(t)\gamma_\theta'(t))dt\\
=&\int_{0}^{r} \left(\mathrm{ct}_k(t)-\mathrm{ct}_k(r)\right) \widehat{\mathrm{Ric}}_k\left(s_k(t)\gamma_\theta'(t)\right)dt.
\end{split}
\end{equation}
From this we obtain \eqref{est2}. Note that $\mathrm{ct}_k(t)-\mathrm{ct}_k(r)$ is positive by Lemma \ref{lem: aux}.

If $k(t)=k$ is a constant, we can express \eqref{eq: sing int} in a more symmetric form. Indeed, we have the ``compound angle formula'' $ s_k(r-t)=s_k(r)c_k(t)-c_k(r)s_k(t)$, so \eqref{eq: sing int} becomes
\begin{equation*}
\begin{split}
\int_{0}^{r}\int_{0}^{\rho} \widehat{\mathrm{Ric}}_k \left(\frac{s_k(t)}{s_k(\rho)}\gamma_\theta'(t) \right) dt \, d\rho
=&\int_{0}^{r} \left(\mathrm{ct}_k(t)-\mathrm{ct}_k(r)\right) \widehat{\mathrm{Ric}}_k\left(s_k(t)\gamma_\theta'(t)\right)dt\\
=&\frac{1}{s_k(r)}\int_{0}^{r} s_k(t)s_k(r-t) \widehat{\mathrm{Ric}}_k(\gamma_\theta'(t))dt.
\end{split}
\end{equation*}

\end{enumerate}
\end{proof}
\begin{remark}
One may compare Theorem \ref{thm: est} \eqref{est1} with \cite[Lem 2.2]{petersen1997relative}, in which the following integral estimate is proved ($k$ is constant):
\begin{align*}
\int_{B_g(r, p)}[(H-\overline H)^+]^{2p} dV \le C(n,p)
\sup_{x\in M} \int_{B_g(r, x)}\left(\widehat {\underline{\mathrm{Ric}}_k}^-\right)^p dV
\end{align*}
Here $p>\frac{n}{2}$, $H=\frac{1}{n-1}\Delta d_p$ is the (normalized) mean curvature of the geodesic sphere $S_g(r, p)$, $\overline H=\frac{s_k'}{s_k}$, and ${f}_{+}$ and $ f_-$ denote the positive and negative part of a function $f$ respectively.
\end{remark}

As there is no curvature assumption on $M$ in Theorem \ref{thm: est}, we are free to choose any comparison function $ s_k$, which gives us much flexibility. The same applies to results in later sections. We notice that a quantity similar to the R.H.S. of \eqref{ineq: lap d1} when $k=0$ was defined in \cite[p. 340]{cheng1975differential} and \cite[p. 202]{yau1975harmonic} to prove a generalized maximum principle.

In many cases, the Laplacian comparison theorem is used to obtain integral estimates for radial functions and as such often a condition weaker than $\int_{0}^{r}\widehat{\mathrm{Ric}}_k\left(\frac{s_k(t)}{s_k(r)}\gamma_\theta'(t)\right)dt\ge 0$ suffices to draw useful conclusions.
We give an instance of this in Proposition \ref{prop: radial}, and will illustrate its applications by Theorem \ref{thm: Cheng first thm} and Theorem \ref{thm: Cheng second thm} as examples.

We say a smooth function $\phi: \mathbb R\to \mathbb R$ is a radial function if $ \phi$ is even.
It is said to be non-increasing if $ \phi'(r)\le 0$ for $ r\ge 0$. We use this terminology because on a Riemannian manifold any smooth radially symmetric function (whenver this makes sense) is of the form $ \phi\circ d_p$ for a radial function $ \phi$ within the injectivity radius of $ p$.

Denote by $B_g(r, p)$ (resp. $ S_g(r, p)$) to be the geodesic ball (resp. geodesic sphere) of radius $ r$ centered at $ p$ in $ (M, g)$, and $ B_{\overline g}(r):=[0, r]\times \mathbb S^{n-1}$ (resp. $ S_{\overline g}(r)$) to be the geodesic ball (resp. geodesic sphere) of radius $ r$ centered at $ 0$ in $ (\overline M, \overline g)$. We use $ \mathcal B_g(r, p)$ to denote the metric ball of radius $ r$ centered at $ p$ in $ M$, so $ \mathcal B_g(r, p):=\{x\in M: d(p, x)<r\}$.
The metric ball of radius $r$ centered at $ 0$ in $ (\overline M, \overline g)$ is also $ B_{\overline g}(r)$ for $ r\le r_0=\min\{r>0: s_k(r)=0\}$, but $ \mathcal B_{g}(r, p)$ may not coincide with $ B_{g}(r, p)$ for large $ r$. Let $\mathrm{c}(\theta)$ be the cut distance in the direction $\theta$.  We also define $\mathcal B_g'(r, p):=\{\exp_p(\rho v): v\in S_pM, \mathrm{c}(v)\ge r\textrm{ and }0\le \rho< r\}\subset \mathcal B_g(r,p)$.

\begin{proposition}\label{prop: radial}
Suppose $\phi, \psi$ are two non-negative radial functions and $ \phi$ is non-increasing.
Suppose
$\displaystyle \int_{\mathcal B_g'(\rho, p)} \widehat{\mathrm{Ric}}_k \left(s_k(t) \partial_t\right) \ge 0$
%\footnote{ $\int_{\mathcal B_g'(\rho, p)} \widehat{\mathrm{Ric}}_k \left(s_k(t) \partial_t\right):=\int_{\mathcal S_g(\rho, p)} \int_{0}^{\rho} \widehat{\mathrm{Ric}}_k \left({s_k(t)}\partial_t \right) dt \, dS$ where $\mathcal S_g(t, p)=\{\exp_p(t \theta): \theta\in S_pM, \mathrm{c}(\theta)> t\}$. }
for all $ 0\le \rho\le r$. Then
\begin{align*}
\int_{\mathcal B_g(r, p)}\langle \nabla (\psi\circ d_p), \nabla (\phi\circ d_p)\rangle \le -\int_{\mathcal B_g(r, p)} \left(\psi\circ d_p\right)\cdot(\overline \Delta \phi)\circ d_p
\end{align*}
where $\overline \Delta \phi(r):=\phi''(r)+\frac{\overline F'(r)}{\overline F(r)}\phi'(r)$ is the Laplacian of $ \phi$ with respect to the metric $ \overline g$.
In particular, within a geodesic ball, in short, $-\int_{B_g(r, p)} \psi \Delta \phi\le -\int_{B_g(r, p)} \psi\overline \Delta \phi$.
\end{proposition}
\begin{proof}
%Let $\mathrm{c}(\theta)$ be the cut distance in the direction $\theta$ and
Let
$a(\theta)=\min \{\mathrm{c}(\theta), r\}$.
Then
\begin{equation}\label{ineq: radial}
\begin{split}
\int_{0}^{a(\theta)}\psi'(t) \phi'(t)F(t, \theta)dt
=&\left[\psi(t)\phi'(t)F(t, \theta)\right]_{t=0}^{a(\theta)}
-\int_{0}^{a(\theta)}\psi\left(\phi''(t)F(t, \theta)+\phi' F'(t, \theta)\right)dt\\
\le&-\int_{0}^{a(\theta)}\psi\left(\phi''F+\phi' F'\right)dt\\
=&-\int_{0}^{a(\theta)}\psi\left[ \phi''+\phi' \frac{\overline F'}{\overline F}+\phi'\cdot \left(\frac{F'}{F}-\frac{\overline F'}{\overline F}\right) \right] Fdt\\
=&-\int_{0}^{a(\theta)}\psi \overline \Delta \phi F dt-\int_{0}^{a(\theta)}\psi
\phi'\cdot \left(\frac{F'}{F}-\frac{\overline F'}{\overline F}\right) Fdt.
\end{split}
\end{equation}
As $\phi'\le 0$, integrating the second term over $S_pM$ and using Theorem \ref{thm: est},
\begin{equation}\label{ineq: radial2}
\begin{split}
-&\int_{S_p M}\int_{0}^{a(\theta)}\psi (t) \phi'(t) \left(\frac{F'(t, \theta)}{F(t, \theta)}-\frac{\overline F'(t)}{\overline F(t)}\right)F(t, \theta) dt\, d\theta\\
\le&
-\int_{0}^{r}\psi (t) |\phi'(t)| \left(\int_{\mathcal S_g(t, p)} \int_{0}^{t} \widehat{\mathrm{Ric}}_k \left(\frac{s_k(\rho)}{s_k(t)}\partial_\rho \right) d\rho \, dS\right) dt\\
=&
-\int_{0}^{r}\psi (t) |\phi'(t)| \frac{1}{s_k(t)^2}\left(\int_{\mathcal B_g'(t, p)}
\widehat{\mathrm{Ric}}_k \left(s_k(\rho) \partial_\rho\right) dV\right) dt\le 0,
\end{split}
\end{equation}
where $\mathcal S_g(t, p):=\{\exp_p(t \theta): \theta\in S_pM, \mathrm{c}(\theta)> t\}$. In view of this, integrating \eqref{ineq: radial} over $S_pM$ will give the result.
\end{proof}

We now prove some generalizations of the Bonnet-Myers theorem. Roughly speaking it says that the weighted integral of the negative part of the Ricci curvature competes with the positive part of the Ricci curvature together with the function $-s_k'/s_k$ to prevent $ M$ from being bounded. An advantage of our result is that we have the flexibility to choose the function $s_k$.
A number of results in the literature, e.g. \cite{galloway1979generalization}, can be reduced to a suitable choice of $s_k$ in the following result, see also \cite{calabi1967ricci}, \cite{avez1972riemannian}, \cite{markvorsen1982ricci}, \cite{cheeger1982finite}, \cite{petersen1998integral}. Ambrose \cite{ambrose1957theorem} also gives a qualitative version (without a diameter bound) involving the integral of the Ricci curvature.

\begin{theorem}\noindent\label{thm: bonnet myers}
\begin{enumerate}
\item\label{item: myers}
Suppose $s_k$ satisfies \eqref{eq: s} such that the smallest positive zero $r_0$ of $s_k$ exists, i.e.
\begin{equation}\label{eq: r0}
r_0:=\min\{r>0: s_k(r)=0\}.
\end{equation}
If
\begin{align}\label{ineq: limsup}
\limsup_{r\to r_0^-} \left[ \frac{1}{s_k(r)^2}\int_{0}^{r}\widehat{\mathrm{Ric}}_{k}(s_k(t)\gamma_\theta'(t))dt-(n-1)\frac{s_k'(r)}{s_k(r)}\right]=\infty
\end{align}
for any $\theta\in S_pM$,
then every geodesic staring from $p$ which is longer than $ r_0$ has a conjugate point, $ d_p \le r_0$ on $ M$, $ M$ is compact and $ \pi_1(M)$ is finite. (See also Remark \ref{rem: rem1}.)

\item
With the same assumption as \eqref{item: myers},
suppose
\begin{align}\label{ineq: limsup2}
\limsup_{r\to r_0^-} \left[ \frac{1}{s_k(r)^2}\fint_{S_pM}\int_{0}^{r}\widehat{\mathrm{Ric}}_k(s_k(t)\partial _t)dt\, d\theta-(n-1)\frac{s_k'(r)}{s_k(r)}\right]=\infty,
\end{align}
then the injectivity radius at $p$ satisfies $ \mathrm{inj}(p) \le r_0$.

\item
Suppose for all $\theta\in S_pM$, there exists a function $s_k$ satisfying \eqref{eq: s} whose smallest positive root $ r_0$ exists, such that
\begin{equation*}
\limsup_{r\to r_0^-} \left[ \frac{1}{s_k(r)^2}\int_{0}^{r}\widehat{\mathrm{Ric}}_k(s_k(t)\gamma_\theta'(t))dt-(n-1)\frac{s_k'(r)}{s_k(r)}\right]=\infty ,
\end{equation*}
then $M$ is compact and $ \pi_1(M)$ is finite.
\end{enumerate}
\end{theorem}

\begin{proof}
\begin{enumerate}
\item
Suppose for the sake of contradiction that $\gamma_\theta:[0, r_0]\to M$ is a geodesic with no point conjugate to $ p$.
Using notation in the proof of Theorem \ref{thm: est}, let $Z_i^r(t):=s_k(t)E_i(t)$ on $ [0, r]$.
Similar to \eqref{eq: IX},
\begin{align*}
\sum_{i=1}^{n-1}I(Z_i^r, Z_i^r)=-\int_{0}^{r}\widehat{\mathrm{Ric}}_k(s_k(t)E_i, \gamma_\theta')dt +(n-1)s_k(r)s_k'(r).
\end{align*}
So in view of \eqref{ineq: limsup}, $\displaystyle \sum_{i=1}^{n-1}I(Z_i^{r_0}, Z_i^{r_0})=\lim_{r\to r_0^-} \sum_{i=1}^{n-1}I(Z_i^r, Z_i^r)\le 0$. This implies $I(Z_i^{r_0}, Z_i^{r_0})\le 0$ for some $ i$.
By the equality case of index lemma, either $Z_i^{r_0}$ is a Jacobi field or there exists a Jacobi field with endpoint values equal to $ Z_i^{r_0}$ with strictly smaller index form, contradicting the assumption that $ \gamma_\theta$ has no conjugate point on $ [0, r_0]$. This implies $ d_p\le r_0$ and $ M$ is compact. By applying the same argument to its universal cover $ \widetilde{M}$, standard covering theory then shows that $ \pi_1(M)$ is finite (\cite[Thm. 11.7]{lee1997riemannian}).

\item
Suppose $\mathrm{inj}(p)>r_0$. Then $d_p$ is smooth on $ S_g(r_0,p)$.
By \eqref{ineq: limsup2}, there exists $\theta \in S_pM$ such that $ \limsup_{r\to r_0^-} \left[ \frac{1}{s_k(r)^2}\int_{0}^{r}\widehat{\mathrm{Ric}}_k(s_k(t)\gamma_\theta'(t))dt-(n-1)\frac{s_k'(r)}{s_k(r)}\right]=\infty
$. Put $ x=\gamma_\theta(r) $ in \eqref{ineq: lap d1}, and taking the limit $ r\to r_0^- $, we get $ \Delta d_p(\gamma_\theta(r_0))=-\infty$, a contradiction.
\item
By \eqref{item: myers} and \cite[Lemma 1]{ambrose1957theorem}, $M$ is compact, and so is $\widetilde M$.
\end{enumerate}
\end{proof}

\begin{remark}\label{rem: rem1}
By noting that $s_k>0$ on $ (0, r_0)$ and $ s_k'(r_0)<0$, we can provide some stronger but finitary conditions alternative to \eqref{ineq: limsup}.
One possibility is that
\begin{align*}
\limsup_{r\to r_0^-} \left[ \frac{1}{s_k(r)}\int_{0}^{r}\widehat{\mathrm{Ric}}_k(s_k(t)\gamma_\theta'(t))dt-(n-1)s_k'(r)\right]>0
\end{align*}
for all $\theta \in S_pM$.
Another simpler (but stronger) condition which clearly indicates the relation with the classical Bonnet-Myers theorem is
\begin{equation}\label{ineq: BM}
\int_{0}^{r_0}\widehat{\mathrm{Ric}}_k \left(s_k(t) \gamma_\theta'(t) \right) dt\ge 0
\end{equation}
for all $\theta \in S_pM$.

To see that the two conditions above are stronger, observe that $s_k'(r_0)<0$. Indeed, by Lemma \ref{lem: aux} \eqref{item1}, since $s_k(r_0)=0$ and $ s_k>0 $ on $ (0, r_0)$, we have $ s_k'(r_0)<0$. From this we have $ \lim_{r\to r_0^-}\frac{s_k'(r)}{s_k(r)}=-\infty$. Then we can see that both the above two conditions imply \eqref{ineq: limsup}.

Similarly we can use the simpler condition $\int_{S_pM}\int_{0}^{r_0}\widehat{\mathrm{Ric}}_k \left(s_k(t) \partial_t\right) dt\, d\theta\ge 0$ to replace \eqref{ineq: limsup2}.
\end{remark}
It is also interesting to see that the existence of a conjugate point implies a condition in terms of $\widehat K_k^1$ similar to \eqref{ineq: BM}. We use $\widehat K^1_k(E(t), \gamma'(t))$ to denote $\widehat K^1_k(\mathrm{span}(E(t)), \gamma'(t))$.
\begin{proposition}\label{prop: conj}
Suppose $\gamma$ is a geodesic of length $r_0$ parametrized by arclength. If $\gamma(r_0)$ is the first conjugate point of $\gamma(0)$ along $\gamma$, then there exists a unit vector field $E(t)$ along $\gamma$, together with functions $k(t)$ and $s_k(t)$ satisfying \eqref{eq: s}, such that $\int_{0}^{r_0}s_k(t)^2\widehat K^1_k(E(t), \gamma'(t))dt\ge 0$ with $r_0=\min\{t>0: s_k(t)=0\}$.

\end{proposition}
\begin{proof}
There exists a nontrivial Jacobi field $Y(t)$ with $Y(0)=0=Y(r_0)$. Let $Y(t)=s(t)E(t)$, where $s(t)=|Y(t)|$ and w.l.o.g. $s'(0)=1$.
We compute $s'=\langle Y', E\rangle $ and
$s''
= -(\langle R(E, \gamma')\gamma', E\rangle -|E'|^2)s =:-k s$.
As $\gamma(r_0)$ is the first conjugate point along $\gamma$,
\begin{align*}
0=I(Y,Y)
=\int_{0}^{r_0} \left(|Y'|^2-\langle R(Y, \gamma')\gamma', Y\rangle\right)
=&\int_{0}^{r_0} \left(s'^2+s^2 |E'|^2-s^2\langle R(E, \gamma')\gamma', E\rangle\right) \\
=&\int_{0}^{r_0} \left(k s^2+s^2 |E'|^2-s^2\langle R(E, \gamma')\gamma', E\rangle\right)\\
\ge&\int_{0}^{r_0} \left(k s^2-s^2\langle R(E, \gamma')\gamma', E\rangle\right)\\
=&-\int_{0}^{r_0}s ^2\widehat K^1_k(E, \gamma').
\end{align*}
\end{proof}

\begin{lemma}[\cite{daicomparison} Lemma 1.4.10]\label{lem: dec}
If $f, g$ are two continuous functions such that $ f(t)/g(t)$ is non-increasing and $ g(t)$ is positive for $ t>0$, then
$\displaystyle \frac{\int_{s}^{r}f(t)dt}{\int_{s}^{r}g(t)dt}$
is non-increasing in $r$ and $ s$.
\end{lemma}

There is also a quantitative version of the lemma.
\begin{lemma}\label{lem: dec2}
If $f(t), g(t)$ are $ C^1$ function, $ a\in \mathbb R$ and $ g(t)>0$, then $h(r)=\frac{\int_{a}^{r}f(t)dt}{\int_{a}^{r}g(t)dt}$ satisfies
$ h'(r)= \frac{g(r)\int_{a}^{r}\left(\int_{a}^{u}g(t) dt \right)\alpha'(u)\, du}{\left(\int_{a}^{r}g(t)dt\right)^2} $
where $ \alpha(t)=\frac{f(t)}{g(t)}$.
\end{lemma}
\begin{proof}
We compute
$\displaystyle h'(r)= \frac{f(r)\int_{a}^{r}g(t)dt -g(r)\int_{a}^{r}f(t)dt}{\left(\int_{a}^{r}g(t)dt\right)^2} $.
By fundamental theorem of calculus,
$\displaystyle \frac{f(r)}{g(r)}-\frac{f(t)}{g(t)}=\int_{t}^{r}\alpha'(u)du$ and so
$\displaystyle f(r)g(t)-g(r)f(t)=g(r)g(t)\int_{t}^{r}\alpha'(u)du$.

Integrating this with respect to $t$ on $ [a, r]$ and using Fubini's theorem,
\begin{align*}
f(r)\int_{a}^{r}g(t)dt-g(r)\int_{a}^{r} f(t)dt
=g(r)\int_{a}^{r}g(t)\int_{t}^{r}\alpha'(u)du \, dt
=g(r)\int_{a}^{r}\alpha'(u)\int_{a}^{u}g(t) dt \, du.
\end{align*}
Therefore
$h'(r)= \frac{g(r)\int_{a}^{r}\alpha'(u)\int_{a}^{u}g(t) dt\, du}{\left(\int_{a}^{r}g(t)dt\right)^2}$.

\end{proof}

We now give the area and volume comparison theorems. We use $|\cdot|$ to denote either the volume of a domain or the area of a hypersurface, whichever makes sense.

The following theorem can be compared to \cite[Thm. 1.1]{petersen1997relative}.
\begin{theorem}\noindent\label{thm: area vol est}
\begin{enumerate}
\item \label{item: vol est}
Suppose $r< \mathrm{inj}(p)$ and
$s_k >0$ on $ (0, r]$.
We have the estimate
\begin{equation}\label{ineq: vol}
|B_g(r, p)| \le \int_{0}^{r} w(\rho)\overline F(\rho)d \rho,
\end{equation}
where $\displaystyle w(\rho)=\int_{S_pM} \exp \left[-\int_{0}^{\rho} \left(\mathrm{ct}_k(t)-\mathrm{ct}_k(\rho)\right) \widehat{\mathrm{Ric}}_k\left(s_k(t)\partial_t\right)dt \right]d\theta$.

In particular, if
\begin{equation}\label{ineq: cond}
\fint_{S_pM} \exp \left[-\int_{0}^{\rho} \left(\mathrm{ct}_k(t)-\mathrm{ct}_k(\rho)\right) \widehat{\mathrm{Ric}}_k\left(s_k(t)\partial_t\right)dt \right] d\theta \le 1
\end{equation}
for $\rho\in [0, r]$, then $|B_g(r, p)|\le | B_{\overline g}(r)|$.

In both cases, the equality holds if and only if $B_g(r, p)$ is isometric to the geodesic ball $ B_{\overline g}(r)$. (Note also that $ \mathrm{ct}_k$ is decreasing by Lemma \ref{lem: aux}.)
\item\label{item: vol mono}
Suppose $r<\mathrm{inj}(p)$ and $ s_k>0$ on $ (0, r]$, then
\begin{equation*}
\frac{d}{dr} \left(\frac{|S_g(r, p)|}{|S_{\overline g}(r)|}\right)
\le -\frac{1}{|B_{\overline g}(r)|} \int_{B_g(r, p)} \widehat{\mathrm{Ric}}_k \left(\frac{s_k(t)}{s_k(r)} \partial_t \right) dV
\end{equation*}
and
\begin{align*}
\frac{d}{dr}\left(\frac{|B_g(r, p)|} {|B_{\overline g}(r)|} \right)
\le -\frac{\overline F(r)}{|B_{\overline g}(r)|^2} \int_{0}^{r} \frac{|B_{\overline g}(u)|}{\overline F(u)} \int_{B_g(u, p)}\widehat{\mathrm{Ric}}_k \left(\frac{s_k(t)}{s_k(u)}\partial_t \right) dV\, du.
\end{align*}
The equality holds if and only if $B_g(r, p)$ is isometric to $ B_{\overline g}(r)$.

In particular, if
$\displaystyle \int_{B_g(\rho, p)} \widehat{\mathrm{Ric}}_k \left(s_k(t) \partial_t \right) dV\ge 0$
for all $\rho\in (0, r)$, then $\frac{| B_g(\rho, p)|}{| B_{\overline g}(\rho)|}$ is non-increasing on $(0, r)$.

\item
For $r\le r_0$ given by \eqref{eq: r0} ($ r_0:=\infty $ if $s_k$ has no positive zero), $\frac{|\mathcal B_g(r, p)|} {|B_{\overline g}(r)|} $ is absolutely continuous and
\begin{align*}
\frac{d}{dr}\left(\frac{|\mathcal B_g(r, p)|} {|B_{\overline g}(r)|} \right)
\le -\frac{\overline F(r)}{|B_{\overline g}(r)|^2} \int_{0}^{r} \frac{| B_{\overline g}(u)|}{\overline F(u)} \int_{\mathcal B_g'(u, p)}\widehat{\mathrm{Ric}}_k \left(\frac{s_k(t)}{s_k(u)}\partial_t \right) dV\, du.
\end{align*}
Assume
$\displaystyle \int_{0}^{r}\widehat {\mathrm{Ric}} \left(s(t) \gamma_\theta'(t) \right) dt\ge 0 $ for any $ \theta \in S_pM $ and any $ r \in(0, r_0) $.
If $|\mathcal B_g(r, p)|=| B_{\overline g}(r)|$ ($ r\le r_0$), then $ \mathcal B_g(r, p)$ is isometric to $ B_{\overline g}(r)$.

\end{enumerate}

\end{theorem}

\begin{proof}[Proof of Theorem \ref{thm: area vol est}]

\begin{enumerate}
\item
As
$|B_g(r, p)|=\int_{0}^{r}\int_{S_pM} F(\rho, \theta) d\theta d\rho$,
the inequality \eqref{ineq: vol} follows directly from Theorem \ref{thm: est} \eqref{est2}.

If the equality holds, then the normal Jacobi fields adapted to $p$ are of the form $s_k(t)e(t)$ for some parallel $e(t)$ orthogonal to $\gamma$. The result can then be proved using the Cartan-Ambrose-Hicks theorem (\cite[Thm. 1.12.8]{klingenberg2011riemannian}), the details are similar to (but simpler than) the proof of Theorem \ref{thm: kahler}, so we omit them here.

\item
Let $A(r)=|S_g(r, p)|=\int_{\mathbb S^{n-1}} F(r, \theta)d\theta$, $\overline A(r)=\int_{\mathbb S^{n-1}}\overline F(r)d\theta=|\mathbb S^{n-1}| \overline F(r)$. Then
\begin{equation}\label{eq: A'}
\begin{split}
\left(\frac{A(r)}{\overline A(r)}\right)'
=&\frac{1}{|\mathbb S^{n-1}|}\int_{\mathbb S^{n-1}}\frac{\partial}{\partial r}\left(\frac{F(r, \theta)}{\overline F(r)}\right)d\theta\\
=&\frac{1}{|\mathbb S^{n-1}|}\int_{\mathbb S^{n-1}} \left(\frac{F'(r, \theta)}{F(r, \theta)}-\frac{\overline F'(r)}{\overline F(r)}\right)\frac{F(r, \theta)}{\overline F(r)} d\theta\\
=&\frac{1}{\overline A(r)}\int_{S_g(r, p)} \left(\frac{F'(r, \theta)}{F(r, \theta)}-\frac{\overline F'(r)}{\overline F(r)}\right) dS.
\end{split}
\end{equation}

So from \eqref{ineq: lap d1},
\begin{equation}\label{ineq: A'2}
\begin{split}
\displaystyle \left(\frac{A(r)}{\overline A(r)}\right)'
\le& -\frac{1}{\overline A(r)} \int_{S_g(r, p)}\int_{0}^{r} \widehat{\mathrm{Ric}}_k \left(\frac{s_k(t)}{s_k(r)} \partial_t \right) dt\, dS\\
=& -\frac{1}{\overline A(r)} \int_{B_g(r, p)} \widehat{\mathrm{Ric}}_k \left(\frac{s_k(t)}{s_k(r)} \partial_t \right) dV.
\end{split}
\end{equation}
Let $V(r)= |B_g(r, p)|=\int_{0}^{r}A(s)ds$
and $\overline V(r)= |B_{\overline g}(r)|=\int_{0}^{r}\overline A(u)du$. Then by Lemma \ref{lem: dec2} and \eqref{ineq: A'2},
\begin{equation}\label{ineq: V'}
\begin{split}
\frac{d}{dr}\left(\frac{V(r)}{\overline V(r)}\right)
=&
\frac{\overline A(r)}{\overline V(r)^2}\int_{0}^{r}\overline V(u)\frac{d}{du}\left(\frac{A(u)}{\overline A(u)}\right)du \\
\le&-\frac{\overline A(r)}{\overline V(r)^2} \int_{0}^{r} \frac{\overline V(u)}{\overline A(u)} \int_{B_g(u, p)} \widehat{\mathrm{Ric}}_k \left(\frac{s_k(t)}{s_k(u)}\partial_t \right) dV\, du.
\end{split}
\end{equation}

\item
Let $\chi(r, \theta): T_pM\to \mathbb R$ be defined by
$\chi(r, \theta)=
\begin{cases}
1, \quad \textrm{if } \mathrm{c}(\theta)>r\\
0, \quad \textrm{otherwise}.
\end{cases}
$
Then \eqref{eq: A'} and \eqref{ineq: A'2} are still true if we replace $F(r, \theta)$ by $\mathcal F(r, \theta):=F(r, \theta)\chi(r, \theta)$, $S_g(r, p)$ by $\mathcal S_g(r, p)=\{\exp_p(r \theta): \theta\in S_pM, \mathrm{c}(\theta)> r\}$, $B_g(r, p)$ by $\mathcal B_g'(r, p)$ and $A(r)$ by $\mathcal A(r)=\int_{S_pM} \mathcal F(r, \theta)d\theta$ which is absolutely continuous.
Let $\mathcal V(r)=\int_0^r \mathcal A(t)dt$, then the analysis in \eqref{ineq: V'} shows that for almost every $r$,
\begin{align*}
\frac{d }{d r}
\left(\frac{\mathcal V(r)}{\overline V(r)}\right)
\le&-\frac{\overline F(r)}{\overline V(r)^2} \int_{0}^{r} \frac{\overline V(u)}{\overline F(u)} \int_{\mathcal B_g'(u, p)} \widehat{\mathrm{Ric}} \left(\frac{s_k(t)}{s_k(u)}\partial_t \right) dV\, du.
\end{align*}
Under the assumption, if $|\mathcal B_g(r, p)|=| B_{\overline g}(r)|$ ($ r\le r_0$), then $ \frac{\mathcal F(r)}{\overline F(r)}=1$ and $\chi(r, \theta) = 1 $ for all $\theta$, so $ \mathcal B_g(r, p)=B_g(r, p)$ and it has the same volume as $ B_{\overline g}(r)$. So from \eqref{item: vol est} and in view of \eqref{eq: sing int}, we conclude that $ \mathcal B_g(r, p)$ is isometric to $ B_{\overline g}(r)$.
\end{enumerate}
\end{proof}

As a simple corollary, we record here a bound for the isoperimetric ratio for geodesic balls, which may be of independent interest.
\begin{proposition}\label{prop: isop}
If $\int_{B_g(\rho, p)}\widehat {\mathrm{Ric}}_k(s_k(t)\partial _t) dV\ge 0$ for all $\rho\in [0, r]$, then $\frac{|B_g(t, p)|}{|S_g(t, p)|} \ge \frac{|B_{\overline g}(t)|}{|S_{\overline g}(t)|} $ for $t\in [0, r]$.
\end{proposition}
\begin{proof}
Using the notation in Theorem \ref{thm: area vol est}, for $0\le s\le t$, we have
$A(s)\ge \frac{A(t) \overline A(s)}{\overline A(t)} $.
Integrate this w.r.t. $s$ on $[0,t]$, we can get the result.
\end{proof}

\begin{remark}\label{rem: sc}
Suppose $k$ is a constant. Our condition for Theorem \ref{thm: area vol est} \eqref{item: vol est} and \eqref{item: vol mono} is much weaker than a Ricci curvature lower bound, but stronger than a scalar curvature lower bound, in the sense that if \eqref{ineq: cond} is satisfied for all small enough $r$, then $R(p)\ge n(n-1)k$. It is not hard to see that
\begin{align*}
\exp \left[-
\int_{0}^{r}
(\mathrm{ct}_k(t)-\mathrm{ct}_k(r))
\widehat{\mathrm{Ric}}_k(s_k(t)\gamma_\theta'(t))dt\right] = 1 - (\mathrm{Ric}(\theta, \theta)-(n-1)k)\frac{r^2}{6}+O(r^3).
\end{align*}
We then have
\begin{equation}\label{eq: expan}
\begin{split}
&\fint_{S_pM} \exp \left[-
\int_{0}^{r} \frac{s_k(t)s_k(r-t)}{s_k(r)} \widehat{\mathrm{Ric}}_k(\gamma_\theta'(t))dt\right] d\theta
=1-\left(\frac{R(p)}{n}-(n-1)k\right)\frac{r^2}{6}+O(r^3)
\end{split}
\end{equation}
where we have used $\displaystyle \fint_{\mathbb S^{n-1}}h(\theta, \theta)d\theta=\frac{\mathrm{tr}(h)}{n}$ for a symmetric bilinear form $ h$.

Therefore we conclude that if \eqref{ineq: cond} is satisfied for any small enough $r>0$, then indeed we have the scalar curvature $R(p)\ge n(n-1)k$.

Theorem \ref{thm: area vol est} combined with \eqref{eq: expan} is also consistent with the fact that the Taylor expansion of the volume of small geodesic balls involve only the scalar curvature at $p$ in the $ (n+2)$-th order (cf. \cite[Theorem 3.1]{gray1974volume}), whereas a global area or volume comparison result is not true assuming only a scalar curvature lower bound. Indeed, by a direct computation, or using the formula in \cite[Theorem 3.1]{gray1974volume}, it can be seen that the volume of a geodesic ball in $ \mathbb H^2\times \mathbb S^2$ (which has zero scalar curvature) is larger than the Euclidean one. More precisely, we have
\begin{align*}
|B_g(r, p)|=b_4 r^4 \left(1+
\frac{1}{72\cdot 6\cdot 8}r^4+O(r^6) \right)
>b_4 r^4
\end{align*}
for $r\approx 0$, where $b_4$ is the volume of the unit ball in $ \mathbb R^4$.
\end{remark}
\subsection{Comparison theorems for distance from a submanifold}

In the following result, we are going to use the (polar) Fermi coordinates with respect to an $\ell$-dimensional submanifold $ \Sigma$ (cf. \cite{gray2012tubes}).
These coordinates are suitable to describe the geometry of the tubular neighborhood of $ \Sigma$.

Recall that if $x$ is within the cut locus of $ \Sigma$, then the Fermi coordinates of $ x$ is $ (r, \theta, z)$, where $ r=d(x, \Sigma)=d(x, z)$ for $ z\in \Sigma$ and $ \gamma_\theta$ is the minimizing geodesic with initial vector $ \theta\in S(N_z\Sigma)$.

The mean curvature vector is defined as $H=\frac{1}{\ell}\sum_{i=1}^\ell \left(\nabla_{e_i} e_i\right)^\perp$, where $e_i$ is an orthonormal frame along $ \Sigma$.

\begin{theorem}\label{thm: submfd}
Suppose $\Sigma$ is an $ \ell$-dimensional submanifold of $ M$. Let $ d_\Sigma: M\to \mathbb R$ be the distance from $ \Sigma$, $ H$ be the mean curvature vector of $ \Sigma$
and $(r, \theta, z)$ be the Fermi coordinates of $ x$. Assume $ s_k>0$ on $ (0, r]$ and that the first zero of $ t\mapsto c_k(t)-\langle H(z), \theta\rangle s (t)$ (if exists) appears no earlier than the cut distance in the direction $\theta$.

\begin{enumerate}
\item
We have
\begin{equation}\label{ineq: lap d Sigma}
\begin{split}
\Delta d_\Sigma (x)\le (\log \overline F)'(r, \theta, z) -\psi(r, \theta, z)
\end{split}
\end{equation}
where
\begin{equation}\label{eq: ov F}
\overline F(t, \theta, z)=\left(c_k(t)+\lambda s_k(t)\right)^\ell s_k(t)^{n-1-\ell},
\end{equation}
\begin{align*}
\psi(r, \theta, z)=\int_{0}^{r} \left(\frac{(c_k(t)+\lambda s_k(t))^2}{(c_k(r)+\lambda s_k(r))^2} \widehat {K}^\ell_k(P^t_\theta (T_z\Sigma), \partial_t)+
\frac{s_k(t)^2}{s_k(r)^2}\widehat {K}^{n-1-\ell}_k(P^t_\theta(\theta^\perp\cap N_z\Sigma), \partial_t)
\right) dt,
\end{align*}
$\lambda=-\langle H(z), \theta\rangle $ and
$P^t_\theta$ is the parallel transport along $ \gamma_\theta$ to $\gamma_\theta(t)$.
\item
The volume element $F(r, \theta, z)$ of $ M$ satisfies
\begin{align*}
F(r, \theta, z)
\le& \exp
\left[-\phi(r, \theta, z)\right]\overline F(r, \theta, z)
\end{align*}
where
\begin{equation}\label{eq: phi}
\begin{split}
\phi(r, \theta, z)
=&
\int_{0}^{r}\left[\left(\frac{s_k(r)}{c_k(r)+\lambda s_k(r)}-\frac{s_k(t)}{c_k(t)+\lambda s_k(t)}\right) \left(c_k(t)+\lambda s_k(t)\right)^2\widehat {K}^\ell_k(P^t_\theta (T_z\Sigma), \partial_t)\right. \\
&
\left. +\left(\mathrm{ct}(t)-\mathrm{ct}(r)\right) s_k(t)^2\widehat {K}^{n-\ell-1}_k(P^t_\theta (\theta^\perp\cap N_z\Sigma), \partial_t)\right]
dt.
\end{split}
\end{equation}

In particular, if $\widehat K^i_k\ge 0$ for $ i= \ell, n-1-\ell $, then $ F(r, \theta, z)\le \overline F(r, \theta, z)$.
\end{enumerate}
\end{theorem}

\begin{remark}
\begin{enumerate}
\item
The condition in Theorem \ref{thm: submfd} is mild.
As shown by Heintze and Karcher in \cite[Cor. 3.3.1]{heintze1978general}, if $k$ is constant, the first focal point appears no later than the first zero of $ c_k(t)+\lambda s_k(t)$ in all space forms of curvature $k$. More generally they show that it holds if $ \widehat {\mathrm{Ric}}_k\ge 0$ when $ \dim \Sigma=n-1$, or if $ \widehat K^1_k \ge 0$ for general $ \ell$. Indeed, from the proof in \cite{heintze1978general} (cf. 3.4.4, which also uses the index lemma), we see that $ \widehat K^\ell_k\ge 0$ suffices.
\item
Similar to Theorem \ref{thm: gunther submfd}, we may instead have a sharper $\overline F$ which is expressed in terms of the second fundamental form of $\Sigma$ (instead of $H$), but then the error term $\psi$ will involve $\ell$ sectional curvatures along $\gamma$ instead of the weaker curvature $\widehat K^\ell_k$. So there is some tradeoff between the two types of estimates. Indeed, in this case, $\overline F(t, \theta, z)=s_k(t)^{n-1-\ell}\det [c_k(t)\mathrm{Id}+s_k(t)A_\theta]$ and
\begin{align*}
&\psi(r, \theta, z)\\
=&\int_{0}^{r}\left(\sum_{i=1}^\ell \frac{(c_k(t)^2+\lambda_i s_k(t))^2}{(c_k(r)^2+\lambda_i s_k(r))^2} \widehat K_k(E_i(t), \partial_t)
+\frac{s_k(t)^2}{s_k(r)^2}\widehat K^{n-1-\ell}_k(P^t_\theta(\theta^\perp \cap N_z \Sigma), \partial_t)\right)dt.
\end{align*}
\end{enumerate}
\end{remark}

\begin{proof}
\begin{enumerate}
\item
Let $E_1, \cdots, E_n=\gamma_\theta(t)$ be a parallel orthonormal frame along $ \gamma_\theta(t)$ such that $ E_1(0), \cdots, E_\ell(0)\in T_z\Sigma$.

For $i=1, \cdots, \ell$, let $Y_i(t)$ be the $ \Sigma$-adapted Jacobi field along $ \gamma$ such that $ Y_i(r)=E_i(r)$. Each tangent space of $ w\in N\Sigma$ is naturally split into the orthogonal direct sum of $ \ell$-dimensional horizontal subspace $ \mathcal H$ and $ (n-\ell)$-dimensional vertical subspace $ \mathcal V$ (tangent space of fiber of $ \pi: N\Sigma\to \Sigma$). Let $ \exp_{N\Sigma}: N\Sigma\to M$ be the exponential map of the normal bundle $ N\Sigma$.
Then the $\ell$-dimensional Jacobian of $ d\exp_{N\Sigma}|_{t\theta}$ along $ \mathcal H|_{t\theta}$ is
$F_{\mathcal H}(t, \theta)
=\frac{\det\left(Y_1(t), \cdots, Y_\ell(t)\right)}{\det\left({Y_1}(0), \cdots, {Y_\ell}(0)\right)} $.
We have the formula (cf. \cite[p. 460]{heintze1978general})
\begin{equation}\label{eq: log FH}
\begin{split}
\left(\log F_{\mathcal H} \right)'(r, \theta)
=&\sum_{i=1}^\ell
\left(\int_{0}^{r}\left(\left\langle {Y_i}'(t), {Y_i}'(t)\right\rangle -\left\langle R (Y_i(t), \partial_t) \partial_t, Y_i(t)\right\rangle\right) dt+
A_{\theta}(Y_i(0), Y_i (0))\right)\\
=&\sum_{i=1}^\ell I_\Sigma(Y_i, Y_i).
\end{split}
\end{equation}
Let $\lambda=-\langle H, \theta\rangle $ and define
$X_i(t)= \frac{c_k(t)+\lambda s_k(t)}{c_k(r)+\lambda s_k(r)}E_i(t)$, $i=1, \cdots, \ell$.
As $X_i(r)=Y_i(r)$ and $ X_i(0)\in T\Sigma$, by the index lemma (\cite[Ch. III, Lemma 2.10]{sakai1996riemannian}), we have
$I_\Sigma(Y_i, Y_i) \le I_\Sigma(X_i, X_i)$.
Similar to \eqref{eq: IX}, integration by parts gives
\begin{equation}\label{eq: IX2}
\begin{split}
&\sum_{i=1}^\ell I_\Sigma(X_i, X_i)\\
=&\sum_{i=1}^\ell\left(\int_{0}^{r}\left(-\langle {X_i}'', X_i\rangle -\langle R (X_i, \partial_t) \partial_t, X_i\rangle \right)dt+\left[\left\langle X_i(t), {X_i}'(t)\right\rangle\right]_{t=0}^r +
A_\theta (X_i(0), X_i(0)) \right) \\
=& -\int_{0}^{r}\frac{(c_k(t)+\lambda s_k(t))^2}{(c_k(r)+\lambda s_k(r))^2} \widehat {K}^\ell_k(P^t_\theta (T_z\Sigma), \partial_t) dt+ \ell\, \frac{c_k'(r)+\lambda s_k'(r)}{c_k(r)+\lambda s_k(r)}.
\end{split}
\end{equation}
Combining \eqref{eq: log FH}, index lemma, \eqref{eq: IX2} and Proposition \ref{prop: hess}, we then have
\begin{equation}
\begin{split}
\label{ineq: hess l}
&\sum_{i=1}^\ell \nabla ^2 d_\Sigma \, (E_i(r), E_i(r))\\
=&(\log F_{\mathcal H})'(r, \theta)
\le -\int_{0}^{r} \frac{(c_k(t)+\lambda s_k(t))^2}{(c_k(r)+\lambda s_k(r))^2} \widehat {K}^\ell_k(P^t_\theta (T_z\Sigma), \partial_t) dt+ \ell\, \frac{c_k'(r)+\lambda s_k'(r)}{c_k(r)+\lambda s_k(r)}.
\end{split}
\end{equation}

We now compute $\sum_{i=\ell+1}^{n-1} \nabla ^2 d_\Sigma(E_i(r), E_i(r))$. Let $u=E_i(r)$, $i=\ell+1, \cdots, n-1$. Then by Gauss lemma \cite[Lemma 2.11]{gray2012tubes}, $\nabla d_\Sigma=\nabla d_z$ along the $ (n-\ell-1)$-sphere tangential to $ E_{\ell+1}(r), \cdots, E_{n-1}(r)$, where $d_z$ in the distance from $z\in \Sigma$ defined on $M$. So at $ \gamma_\theta(r)$,
\begin{equation}\label{eq: two hess}
\begin{split}
\nabla ^2 d_\Sigma (u, u)
=\langle \nabla_u (\nabla d_\Sigma), u\rangle
=-\langle \nabla d_\Sigma, \nabla_u u\rangle
=-\langle \nabla d_z, \nabla_u u\rangle
= \nabla^2 d_z(u, u).
\end{split}
\end{equation}
So by \eqref{eq: IX} and Proposition \ref{prop: hess},
\begin{equation}\label{ineq: hess2}
\sum_{i=\ell+1}^{n-1}\nabla ^2 d_\Sigma(E_i(r), E_i(r))
\le -\int_{0}^{r}\frac{s_k(t)^2}{s_k(r)^2} \widehat {K}^{n-\ell-1}_k(P^t_\theta(v^\perp\cap N_z\Sigma), \partial_t)dt+(n-\ell-1)\frac{s_k'(r)}{s_k(r)}.
\end{equation}

As $\nabla ^2 d_\Sigma (\partial_t, \partial_t)=0$, adding the above inequality to \eqref{ineq: hess l}, we get \eqref{ineq: lap d Sigma}.

\item

Using
$F=\sqrt{\det (g_{ij})}$ in Fermi coordinates, we see
that $\Delta d_\Sigma (x)= \frac{F'(r, \theta, z)}{F(r, \theta, z)}$.
So integrating \eqref{ineq: lap d Sigma} and using $\left(\frac{s_k(r)}{c_k(r)+\lambda s_k(r)}\right)'=\frac{1}{(c_k(r)+\lambda s_k(r))^2}$ (by Lemma \ref{lem: aux}), we can proceed as in \eqref{eq: sing int} to show that $F(r, \theta, z) \le \exp \left[-\phi(r, \theta, z)\right]\overline F(r, \theta, z)
$ where $\phi$ is given by \eqref{eq: phi}.
We also see from the proof that the integrand in $\phi$ is non-negative if $ \widehat K^\ell_k$ and $ \widehat{K}^{n-1-\ell}_k$ are non-negative.
\end{enumerate}
\end{proof}
If $r$ is smaller than the injectivity radius $\mathrm{inj}(\Sigma)$ of $ \Sigma$, we define
$S(r, \Sigma):=\{x\in M: d_\Sigma(x)=r\}$ and $ B(r, \Sigma):=\{x\in M: d_\Sigma(x)<r\}$.
We have the following area and volume comparison theorems. We impose more conditions than Theorem \ref{thm: submfd} to make the statement cleaner.
\begin{theorem}\label{thm: submfd vol}
Suppose $\Sigma$ is an $ \ell$-dimensional submanifold of $ M$ and let $ dS$ be the measure on $ \Sigma$. Define
$\overline A(r)=\int_\Sigma\int_{S(N_z\Sigma)} \overline F(r, \theta, z)d\theta dz$ and $\overline V(r)=\int_{0}^{r}\overline A(t)dt$ with $\overline F$ in \eqref{eq: ov F}.
Assume $ k$ is constant, $ s_k>0$ on $ (0, r]$, $ \widehat{K}_k^ \ell \ge 0$ and $ \widehat{K}_k^{ n-\ell-1 }\ge 0$.
\begin{enumerate}
\item \label{item: submfd1}
We have
\begin{align*}
|S(r, \Sigma)|
\le& |\mathbb S^{n-\ell-1}|s_k(r)^{n-\ell-1}\int_\Sigma f(r, |H|)\fint_{S(N_z\Sigma)}\exp(-\phi) d\theta dS\\
\le& |\mathbb S^{n-\ell-1}| s_k(r)^{n-\ell-1}\int_\Sigma f(r, |H|)dS
\end{align*}
where $f(r, h):=\int_{\mathbb S^{n-\ell-1}} \left(c_k(r)+ \langle h e_0, \theta\rangle \right) ^\ell d\theta$ with $ e_0=(1, 0, \cdots, 0)\in \mathbb S^{n-\ell-1}$, and $ \phi\ge 0$ is given in \eqref{eq: phi}.
\item
The function $f(r, h)$ in \eqref{item: submfd1} is non-decreasing in $ h$.
In particular, if $|H|\le h_0$ for some constant $ h_0$, then
$|S(r, \Sigma)|\le |\mathbb S^{n-\ell-1}| s_k(r)^{n-\ell-1}f(r, h_0) |\Sigma|$.
\item\label{item: submfd3}
Suppose $(\log \overline F)'(t, \theta, z)\ge 0 $ for $t\in(0, r_0)$ and $(\theta, z)\in S(N\Sigma)$.
Then on $[0, r_1]$, $\overline A(r)-|S_g(r, \Sigma)|$ is non-negative and non-decreasing, and $\overline V(r)-|B_g(r, \Sigma)|$ is non-negative, non-decreasing and convex. Here $r_1=\min\{r_0, \mathrm{inj}(\Sigma) \}$.
\item
Suppose $\Sigma$ is a minimal submanifold, i.e. $ H=0$, then
\begin{align*}
\frac{d}{dr}\left(\frac{|S(r, \Sigma)|}{\overline A(r)}\right)
\le -\frac{1}{\overline A(r)} \int_ {S(r, \Sigma)} \phi(x) dS
\le 0,
\end{align*}
Here $\phi(x):=\phi(r(x), \theta(x), z(x))$.
We also have
\begin{align*}
\frac{d}{dr}
\left(\frac{|B(r, \Sigma)|}{\overline V(r)}\right)
\le-\frac{\overline A(r)}{\overline V(r)^2} \int_{0}^{r} \frac{\overline V(u)}{\overline A(u)} \int_{S(u, \Sigma)}\phi(x) dS \, du\le 0.
\end{align*}

\end{enumerate}
\end{theorem}

\begin{proof}
We use the notation in Theorem \ref{thm: submfd}.
\begin{enumerate}
\item
We have $|S(r, \Sigma)|=\int_{\Sigma}\int_{S(N_z\Sigma)} F(r, \theta, z)d\theta dS$.

Define $f(r, h):=\int_{\mathbb S^{n-\ell-1}} \left(c_k(r)+h \langle e_0, \theta\rangle \right) ^\ell d\theta$. Then it follows by the invariance of the spherical measure that $\int_{S(N_p\Sigma)} \overline F(r, \theta, z)d\theta=f(r, |H(z)|) s_k(r)^{n-\ell -1} $. From this and Theorem \ref{thm: submfd} we have
\begin{align*}
|S(r, \Sigma)|
\le& \int_\Sigma f(r, |H(z)|)s_k(r)^{n-\ell-1}\int_{S(N_z\Sigma)}\exp[-\phi(r, \theta, z)]d\theta dS\\
\le& |\mathbb S^{n-\ell-1}| s_k(r)^{n-\ell-1}\int_\Sigma f(r, |H(z)|)dS.
\end{align*}
\item
The fact that $f(r, h)$ is non-decreasing in $ h$ is proved using the same idea in \cite[Proposition 2.1.1]{heintze1978general}.
\item
By Theorem \ref{thm: submfd}, we have
$F'(r, \theta, z)
\le {\overline F'(r, \theta, z)}\frac{F(r, \theta, z)}{\overline F(r, \theta, z)}$.
So if $\overline F'(r, \theta, z)\ge 0$, then by Theorem \ref{thm: submfd} again,
$F'(r, \theta, z) \le {\overline F'(r, \theta, z)} \exp[-\phi(r, \theta, z)]\le \overline F'(r, \theta, z)$.
This implies
\begin{align*}
\frac{d}{dr}|S_g(r, \Sigma)|=
\frac{d}{dr}\left(\int_\Sigma\int_{S(N_z\Sigma)} F(r, \theta, z)d\theta dz\right)\le
\frac{d}{dr}\left(\int_\Sigma\int_{S(N_z\Sigma)} \overline F(r, \theta, z)d\theta dz\right)=\overline A'(r).
\end{align*}
The properties of $\overline V(r)-|B_g(r, \Sigma)|$ follow from this.

We remark that if there exists $\overline \Sigma$ in $\overline M_k$ with the same mean curvature vector $H$ as $\Sigma$, then $\frac{d}{dt}(\log \overline F)$ is the mean curvature of $S_{\overline g}(t, \overline \Sigma)$.
\item
Suppose $H=0$, then
$(\log F)'(r, \theta, z)
\le -\phi(r, \theta, z)+(\log \overline F)'(r)
$ where $ \overline F(t)=c_k(t)^\ell s_k(t)^{n-\ell -1}$.

We have $\overline A(r)=|\mathbb S^{m-\ell-1}||\Sigma|\overline F(r)$ and we can then proceed as in \eqref{ineq: A'2} to show that \begin{align*}
\frac{d}{dr}\left(\frac{S(r, \Sigma)}{\overline A(r)}\right)
\le& -\frac{1}{\overline A(r)}\int_\Sigma\int_{N_z\Sigma} \phi(r, \theta, z) F(r, \theta, z) \, d\theta\, dz\\
=& -\frac{1}{\overline A(r)} \int_ {S(r, \Sigma)} \phi(r(x), \theta(x), z(x)) dS \le 0.
\end{align*}
Note that $\overline A(r)>0$ by \eqref{item: submfd1}.
The second inequality is similar to \eqref{ineq: V'}.
\end{enumerate}
\end{proof}

\section{Comparison results in K\"{a}hler manifolds}\label{sec: kahler}
\subsection{Comparison model and notions of curvatures}
We now turn to a K\"{a}hler manifold $M$ whose complex structure is given by $J$ with $\dim_{\mathbb C}M=n$. We will study it from the real differential geometric point of view, i.e. regard it as a Riemannian manifold with a parallel tensor $J$. It turns out that the model space that we are comparing is not a warped product space. Indeed, the complex space forms $\overline M_k$ of constant holomorphic sectional curvature $ k$ (complex Euclidean space $ \mathbb C^n$, complex projective space $ \mathbb {CP}^n$ and complex hyperbolic space $ \mathbb {CH}^n$) are not warped products unless $ k= 0$ or $ n=1$. See \eqref{eq: H} for the definition of the holomorphic sectional curvature. We now rewrite the metric of $ \overline M_k$ in a form which is comparable to the warped product metric in Subsection \ref{subsect: Riem}.

First of all, let $\mathbb S^{2n-1}=\{z\in \mathbb C^n: |z|=1\}$.
There is a natural $\mathbb S^1$-action on $ \mathbb S^{2n-1}$ by $ e^{i\theta}\cdot(z_1, \cdots, z_n)=(e^{i\theta} z_1, \cdots, e^{i\theta}z_n)$, $ i=\sqrt{-1}$. This action induces a splitting of $ T\mathbb S^{2n-1}=\mathcal H\oplus\mathcal V$ where $ \mathcal V$ is the tangent space of the fiber and $ \mathcal H$ is its orthogonal complement with respect to standard round metric. Let $ g_{\mathcal V}$ and $ g_{\mathcal H}$ be the induced metric from the round metric onto the vertical and horizontal space respectively.

We claim that if $k$ is constant and $ s=s_k$, then within the cut locus, the Riemannian metric of $ \overline M_k$ is of the form
\begin{equation}\label{eq: g cpx}
\overline g=dt^2+ s(t)^2 g_{\mathcal H}+\frac{1}{4}s(2t)^2 g_{\mathcal V}
\end{equation}
defined on $[0, t_0)\times \mathbb S^{2n-1}$,
where $t_0$ is the first positive zero of $ s(2t)$.

This is clear for $k=0$.
We illustrate this for $k=1$. Let $U_0=\{[1, z_1, \cdots, z_n]\in \mathbb CP^n\}\cong \mathbb C^n$. In this coordinate neighborhood, the Fubini-Study metric is of the form
\begin{equation}\label{fs}
\begin{split}%
\overline g=&\frac{dz\cdot d\overline z}
{1+|z|^2}
-\frac{(\overline {z}\cdot dz)(z\cdot d\overline {z})}{\left(1+|z|^2\right)^2}.
\end{split}
\end{equation}
where $z=(z_1, \cdots, z_n)$ and $ z\cdot w=\sum_{j=1}^n z_j w_j$.
From \eqref{fs}, the distance from $[1, 0, \cdots, 0]$ with respect to $ \overline g$ is given by $ t=\int_{0}^{|z|}\frac{1}{1+r^2}dr=\tan^{-1}(|z|)$. So
writing $z=(\tan t) w$, $w\in \mathbb S^{2n-1}\subset \mathbb C^n$ gives
$\overline g=dt^2+\sin^2t \, dw\cdot d\overline {w}-\sin^4 t(w\cdot d\overline {w})(\overline {w}\cdot dw)$,
where we have used $w\cdot d\overline {w}+\overline {w}\cdot dw=0$.

Along the vertical fiber, a unit tangent vector (w.r.t. $g_{\mathbb S^{2n-1}}$) can be represented by $v=\sqrt{-1}w$ at the point $ (t, w)$, and so we have
$|v|_{\overline g}^2=\sin^2 t -\sin^4 t=\frac{1}{4}\sin^2 (2t)$.

Along the horizontal fiber, a unit tangent vector $v$ (w.r.t. $ g_{\mathbb S^{2n-1}}$) is orthogonal to both $ w$ and $ iw$, so $ 0=w\cdot \overline v+\overline {w}\cdot v$ and $ 0=i w\cdot \overline v-i\, \overline {w}\cdot v$.
This implies $w\cdot d\overline w=\overline w\cdot dw=0$ and so $ |v|_{\overline g}^2=\sin^2 t$.

Similarly, this holds for the complex hyperbolic metric $\overline g=\frac{|dz|^2}{1-|z|^2}+\frac{| \overline z\cdot dz |^2}{(1-|z|^2)^2}$, $|z|<1$.

In view of \eqref{eq: g cpx}, it is natural to propose the following model. Let $k_j(t)$, $j=1, 2$, be two continuous functions and suppose $s_{k_j}(t)$ satisfies \eqref{eq: s} for $ k=k_j$.
We define on $\overline M=[0, t_0)\times \mathbb S^{2n-1}$ the metric
\begin{align}\label{eq: model}
\overline g=dt^2+s_{k_1}(t)^2 g_{\mathcal H} +s_{k_2}(t)^2 g_{\mathcal V}
\end{align}
where $t_0$ is the first positive zero of $ s_{k_1}(t) s_{k_2}(t)$.

By a computation similar to \cite[Ch 7, Proposition 35]{o1983semi}, we see that $\overline \Delta r=(2n-2)\frac{s_{k_1}'(r)}{s_{k_1}(r)}+\frac{s_{k_2}'(r)}{s_{k_2}(r)}$ where $ r$ is the $\overline g$-distance from $ 0$ and $ \overline \Delta$ is the Laplacian of $ \overline g$. Similar calculations as in \cite[Ch 7, Proposition 42]{o1983semi} also implies that along the geodesic $ \theta\mapsto t \theta$, $ R_{\overline g}(u, \partial_t)\partial_t=k_1 u$ for $ u\in \mathcal H$, $ R_{\overline g}(u, \partial_t)\partial_t=k_2 u$ for $ u\in \mathcal V$.

We endow a complex structure on $\overline M$ as follows.
Let $r(t)$ be the solution to
$\frac{r'(t)}{r(t)}=\frac{1}{s_{k_2}(t)}$
with $r(0)=0$ and define a diffeomorphism $ \overline M=[0, t_0)\times \mathbb S^{2n-1}\to B(r(t_0), 0)\subset \mathbb C^n$ by $ \Phi(t, w)=r(t)w$. Let $ J=J_{\overline M}$ be the complex structure on $ \overline M$ induced from $ \mathbb C^n$ via this map.

\begin{lemma}\label{lem: parallel}
The complex structure $J$ is parallel w.r.t. $ \overline g$ in \eqref{eq: model} along the $\overline g$-geodesic $ t\mapsto tw_0$, where $ w_0\in \mathbb S^{2n-1}$.
\end{lemma}
\begin{proof}
Regard $\mathbb S^{2n-1}=\{w\in \mathbb C^n: |w|=1\}$ and so identify $ T_{w_0} \mathbb S^{2n-1}=\{z\in \mathbb C^n: \mathrm{Re}(z\cdot \overline w_0)=0\}$.
We can choose a normal coordinates (w.r.t. round metric) $\{\theta^j\}_{j=1}^{2n-1}$ around $ w_0\in \mathbb S^{2n-1}$ such that at $ w_0$, $ \frac{\partial}{\partial \theta^{2n-1}}=\sqrt{-1}w_0$ and $ \sqrt{-1}\frac{\partial}{\partial \theta^j}=\frac{\partial}{\partial \theta^{n-1+j}}$ for $ j=1, \cdots, n-1$. Then it is easy to see that along $ \gamma(t)=tw_0$, $ J\left(\frac{\partial}{\partial t}\right)=\frac{1}{s_{k_2}(t)}\frac{\partial}{\partial \theta^{2n-1}}$, $ J\left(\frac{\partial}{\partial \theta^{2n-1}}\right)=-s_{k_2}(t)\frac{\partial}{\partial t}$,
$J(\frac{\partial}{\partial \theta^j})=\frac{\partial}{\partial \theta^{n-1+j}}$, and $J\left(\frac{\partial}{\partial \theta^{n-1+j}}\right)=-\frac{\partial}{\partial \theta^j}$ for $ j=1, \cdots, n-1$.

From this, it then suffices to show that along $\gamma$, the vector fields $\frac{\partial}{\partial t}$, $\frac{1}{s_{k_2}(t)}\frac{\partial}{\partial \theta^{2n-1}}$ and $ \frac{1}{s_{k_1}(t)}\frac{\partial}{\partial \theta^j}$ are parallel for $ j=1, \cdots, 2n-2$. This can be verified directly by using \eqref{eq: model} and the Koszul formula (note that $ \frac{\partial}{\partial \theta^{2n-1}}\in \mathcal V$ and $ \frac{\partial}{\partial \theta^j}\in \mathcal H$, $ j=1, \cdots, 2n-2$):
\begin{align*}
2\langle \nabla_{\partial_t}Y, Z\rangle
=
\partial_t\langle Y, Z\rangle +Y\langle Z, \partial_t\rangle -Z\langle \partial_t, Y\rangle
-\langle Y, [\partial_t, Z]\rangle -\langle Z, [Y, \partial_t]\rangle +\langle \partial_t, [Z, Y]\rangle.
\end{align*}
\end{proof}

Now we recall some notions of curvatures.
Given two $J$-invariant planes $ \Pi_1$ and $ \Pi_2$ in $ T_pM$, we define the bisectional curvature to be
\begin{align*}
\mathrm{B}(\Pi_1, \Pi_2):=\langle R(v_1, Jv_1)J v_2, v_2\rangle
\end{align*}
where $v_i$ is a unit vector in $ \Pi_i$.
We also define the holomorphic sectional curvature
for $v\ne 0$ to be
\begin{align}\label{eq: H}
{\mathrm{H}}(v):=\frac{\langle R(v, Jv)Jv, v\rangle}{|v|^2}.
\end{align}
The reason for the denominator $|v|^2$ is to give a consistent notion of $ \ell$-holomorphic sectional curvature defined later and is immaterial, as we often only consider the case where $ |v|=1$. For a $ J$-invariant two-plane $ \Pi=\mathrm{span}(v, Jv)$, we define $ {\mathrm{H}}(\Pi)={\mathrm{H}}(v)$ for $ |v|=1$.
It is straightforward to check that $\mathrm{B}(\Pi_1, \Pi_2)$ and $ {\mathrm{H}}(\Pi)$ are well-defined.

We say the bisectional curvature $M$ is bounded below by $ k$ if for any $ p\in M$ and any $ u, v\in T_p M$,
\begin{align}\label{ineq: bisec}
\langle R(u, Ju)Jv, v\rangle
\ge k \left(|u|^2|v|^2+\langle u, v\rangle^2+\langle u, J v\rangle^2\right).
\end{align}
Note that our convention is one-half of that in \cite{goldberg1967holomorphic} but consistent with \cite{li2005comparison}. E.g. the Fubini-Study metric $ g_{FS}=\frac{dz d\overline z}{(1+|z|^2)^2}=\frac{1}{4}g_{\mathbb S^2}$ of $ \mathbb {CP}^1$ has constant holomorphic sectional curvature $ 2$ (and constant sectional curvature $ 4$).

We say the holomorphic sectional curvature of $M$ is bounded below by $k$ if
\begin{align*}
{\mathrm{H}}(v)\ge 2k|v|^2
\end{align*}
for any $ p\in M$ and $ v\in T_pM$.
Note that this is weaker than \eqref{ineq: bisec}.
We also define the orthogonal Ricci curvature (\cite{ni2018comparison}) to be
\begin{align*}
\mathrm{Ric}^\perp (v, v)=\mathrm{Ric}(v, v)-\frac{1}{|v|^2}R(v, Jv, Jv, v)
\end{align*}
for $0\ne v\in T_pM$.

As in the Riemannian case, for a function $k(t)$ we define \begin{equation*}
\widehat{\mathrm{Ric}}_k^{\perp}(v):= \mathrm{Ric} ^{\perp}(v, v)-2(n-1)k g(v, v) \quad\textrm{ and }\quad
\widehat{\mathrm{H}}_k(v):= {\mathrm{H}}(v)-2k g(v, v).
\end{equation*}

For complex space forms $\overline M_k$, we have $\widehat{\mathrm{Ric}}^\perp_{k}=0$ and $ \widehat {\mathrm{H}}_{2k}=0$.
In \cite{li2005comparison}, among other things, Li and Wang proved volume comparison and Laplacian comparison theorems for K\"{a}hler manifolds under a lower bound of the bisectional curvature. The comparison spaces are the complex space forms. Ni and Zheng \cite{ni2018comparison} improved their results by relaxing the condition to a lower bound of the orthogonal Ricci curvature and holomorphic sectional curvature. We will show a generalization of these results under an integral bound of a mixture of the orthogonal Ricci curvature and the holomorphic sectional curvature.
\subsection{Comparison theorems for distance from a point}

In the following,
$\Delta=\mathrm{tr}(\nabla ^2)$ refers to the Laplacian w.r.t. the Riemannian metric, which is $-2$ times the value of $\overline \partial ^*  \overline \partial $ on smooth functions.
\begin{theorem}\label{thm: kahler}
Let $(M, g)$ be a K\"{a}hler manifold. We have the following estimates.
\begin{enumerate}
\item
Assume there is no cut point of $p$ along $ \gamma_\theta$ on $ [0, r]$. If $ s_{k_i}(r)>0$, then
\begin{align*}
\Delta d(x)
\le(2n-2)\frac{s_{k_1}'(r)}{s_{k_1}(r)} + \frac{s_{k_2}'(r)}{s_{k_2}(r)}-\phi(r, \theta)
= \overline \Delta \, \overline d\;(r)-\phi(r, \theta)
\end{align*}
where $x=(r, \theta)$ in geodesic polar coordinates, $ \overline \Delta\, \overline d\;(r)$ is the corresponding Laplacian of the distance function w.r.t. $ \overline g$, and
\begin{equation*}
\phi(r, \theta)= \int_{0}^{r}\left(\frac{s_{k_1}(t)^2}{s_{k_1}(r)^2} \widehat{\mathrm{Ric}}_{k_1}^\perp (\gamma_\theta'(t)) + \frac{s_{k_2}(t)^2}{s_{k_2}(r)^2} \widehat {\mathrm{H}}_{\frac{k_2}{2}}(\gamma_\theta'(t)) \right)dt.
\end{equation*}
\item\label{item: cpx2}
Suppose $r< \mathrm{inj}(p)$ and $ s_{k_i}>0$ on $ (0, r]$. Then
\begin{align*}
|S_g(r, p)|\le |S_{\overline g}(r)|
\fint_{S_pM}
\exp \left[-\psi(r, \theta)\right]d\theta
\end{align*}
where
\begin{align*}
\psi(r, \theta):=\int_{0}^{r} \left(\left(\mathrm{ct}_{k_1}(t)-\mathrm{ct}_{k_1}(r)\right) \widehat{\mathrm{Ric}}_{k_1}^\perp\left(s_{k_1}(t) \gamma_\theta'(t)\right)+\left(\mathrm{ct}_{k_2}(t)-\mathrm{ct}_{k_2}(r)\right) \widehat {\mathrm{H}}_{\frac{k_2}{2}}(s_{k_2}(t)\gamma_\theta'(t))\right) dt.
\end{align*}
We also have
\begin{equation*}
|B_g(r, p)| \le \int_{0}^{r} w(\rho)\overline F(\rho)d \rho,
\end{equation*}
where $w(\rho)=\int_{S_pM} \exp\left[-\psi(\rho, \theta)\right]d\theta $ and $ \overline F(r)=s_{k_1}(r)^{2n-2}s_{k_2}(r)$. In particular, if $ \psi\ge 0$, then $ |B_g(r, p)|\le |B_{\overline g}(r)|$.

\item\label{item: cpx3}
Suppose
$r<\mathrm{inj}(p)$ and $ s>0$ on $ (0, r]$, then
\begin{equation*}
\frac{d}{dr} \left(\frac{|S_g(r, p)|}{|S_{\overline g}(r)|}\right)
\le -\frac{1}{|S_{\overline g}(r)|} \int_{B_g(r, p)}
\left(\frac{s_{k_1}(t)^2}{s_{k_1}(r)^2} \widehat{\mathrm{Ric}}_{k_1}^\perp (\partial _t) + \frac{s_{k_2}(t)^2}{s_{k_2}(r)^2} \widehat {\mathrm{H}}_{\frac{k_2}{2}}(\partial _t) \right)
dV
\end{equation*}
and
\begin{align*}
&\frac{d}{dr}\left(\frac{|B_g(r, p)|} {|B_{\overline g}(r)|} \right)\\
\le& -\frac{\overline F(r)}{|B_{\overline g}(r)|^2} \int_{0}^{r} \frac{|B_{\overline g}(u)|}{\overline F(u)} \int_{B_g(u, p)}
\left(\frac{s_{k_1}(t)^2}{s_{k_1}(u)^2} \widehat{\mathrm{Ric}}_{k_1}^\perp (\partial _t) + \frac{s_{k_2}(t)^2}{s_{k_2}(u)^2} \widehat {\mathrm{H}}_{\frac{k_2}{2}}(\partial _t) \right)
dV \, du.
\end{align*}

on $(0, r)$.

\item
In \eqref{item: cpx2}, \eqref{item: cpx3}, the equality holds if and only if $B_g(r, p)$ is isometric to $ B_{\overline g}(r)$. The isometry is holomorphic if $\overline g$ is K\"ahler.

\end{enumerate}
\end{theorem}
\begin{proof}
Again let $\gamma(t)$ be a unit speed geodesic. We can define a parallel orthonormal frame along $ \gamma(t)$ of the form
$\{F_1, \cdots, F_{2n}\}=\{e_1, Je_1, \cdots, e_{n-1}, J e_{n-1}, e_n, J e_n=\gamma'(t)\}$. As before, let $Y_i(t)$, $i=1, \cdots, 2n-1$, be the Jacobi field with the endpoint values $Y_i(0)=0$ and $ Y_i(r)=F_i(r)$.

Let $X_i(r)= \frac{s_{k_1}(t)}{s_{k_1}(r)}F_i(t)$, $i=1, \cdots, 2n-2$ and $ X_{2n-1}(t)=\frac{s_{k_2}(t)}{s_{k_2}(r)}F_{2n-1}(t)$. Then $ X_i(0)=Y_i(0)$ and $ X_i(r)=Y_i(r)$.
By \eqref{eq: log F} and \eqref{ineq: index},

\begin{align*}
&(\log F)'(r, \theta)\\
\le&\sum_{i=1}^{2n-1}\int_{0}^{r}\left(-\langle {X_i}'', X_i\rangle -\langle R (X_i, \gamma') \gamma', X_i\rangle \right)dt+\sum_{i=1}^{2n-1}\langle X_i(r), {X_i}'(r)\rangle\\
=&\sum_{i=1}^{2n-2}\int_{0}^{r}\frac{s_{k_1}(t)^2}{s_{k_1}(r)^2} (k_1-\langle R(F_i, \gamma'(t))\gamma'(t), F_i\rangle)dt
+ \int_{0}^{r}\frac{s_{k_2}(t)^2}{s_{k_2}(r)^2}(k_2-\langle R(e_n, \gamma'(t))\gamma'(t), e_n \rangle)dt\\
&+(2n-2)\frac{s_{k_1}'(r)}{s_{k_1}(r)}+\frac{s_{k_2}'(r)}{s_{k_2}(r)}\\
=&\sum_{i=1}^{n-1}\int_{0}^{r}\frac{s_{k_1}(t)^2}{s_{k_1}(r)^2} (2k_1-\langle R(e_i, \gamma'(t))\gamma'(t), e_i\rangle-\langle R(Je_i, \gamma'(t))\gamma'(t), Je_i\rangle)dt\\
&+ \int_{0}^{r}\frac{s_{k_2}(t)^2}{s_{k_2}(r)^2}(k_2-\langle R(e_n, Je_n)Je_n, e_n \rangle)dt +(2n-2)\frac{s_{k_1}'(r)}{s_{k_1}(r)}+\frac{s_{k_2}'(r)}{s_{k_2}(r)}\\
=&-\int_{0}^{r}\left(\frac{s_{k_1}(t)^2}{s_{k_1}(r)^2}\left(\mathrm{Ric}^\perp (\partial_t)-2(n-1)k_1\right)
+\frac{s_{k_2}(t)^2}{s_{k_2}(r)^2}\left({\mathrm{H}}(\partial_t)-k_2\right) \right) dt\\
&+(2n-2)\frac{s_{k_1}'(r)}{s_{k_1}(r)} + \frac{s_{k_2}'(r)}{s_{k_2}(r)}.
\end{align*}
Notice that the term at the last line is $\overline \Delta \, \overline d\;(r)$.
Except the equality case, the proof then proceeds as in Theorem \ref{thm: est} and Theorem \ref{thm: area vol est}.

For the equality case,
let $\iota: T_pM \to T_0 \overline M$ be a holomorphic isometry (i.e. $ \iota\, \circ J_M=J_{\overline M}\circ \, \iota$).
Let $\phi=\exp_0\circ\, \iota\circ\exp_p^{-1}: B_g(r, p)\to B_{\overline g}(r)$. To show that it is an isometry, by the Cartan-Ambrose-Hicks theorem (\cite[Thm. 1.12.8]{klingenberg2011riemannian}), it suffices to show that the map $\iota_t$ defined by $ \overline P^t_{\overline \theta}\circ \iota\circ \left(P^t_\theta\right)^{-1}$ satisfies
\begin{align}\label{eq: CAH}
\iota_t\left(R_g(u, \gamma'(t))\gamma'(t)\right)
= R_{\overline g}(\iota_t(u), \overline \gamma'(t))\overline \gamma'(t).
\end{align}
Here $P^t_\theta$ is the parallel transport along $ \gamma=\gamma_\theta(t)$ from $ T_pM$ to $ T_{\gamma(t)}M$, $ \overline \theta=\iota (\theta)$ and $ \overline P^t_{\overline \theta}$ is the corresponding parallel transport along the geodesic $ \overline \gamma=\overline \gamma_{\overline \theta}$ with initial vector $ \overline \theta$ in $ \overline M$.

Suppose $u=J \gamma'(t)$, then as $s_{k_2}(t)u$ is a Jacobi field by index lemma, from the Jacobi field equation we have $ R_g(u, \gamma'(t))\gamma'(t)=k_2 u$. Recall that $ J_{\overline M}$ is parallel along $ \overline \gamma $ by Lemma \ref{lem: parallel}. So we see that L.H.S. of \eqref{eq: CAH} is $ k_2\iota_t(J_M\gamma'(t))=k_2 J_{\overline M} \overline \gamma '(t) = R_{\overline g}(J_{\overline M}\overline \gamma '(t), \overline \gamma '(t))\overline \gamma '(t)$, which is the R.H.S.. Similarly \eqref{eq: CAH} holds for $ u= F_i$ for $ i=1, \cdots, 2n-2$.

By \cite[Lem. 2.5]{tam2012some}, this isometry is holomorphic if $\overline g$ is K\"ahler.

\end{proof}

The case where $k_1=k=\mathrm{const}$ and $ k_2=4k$ generalizes \cite[Cor. 1.3]{ni2018comparison}. In this case, $ s_{k_1}(t)=s_k(t)$ and $ s_{k_2}(t)=\frac{1}{2}s_k(2t)$. Indeed, the proof of Theorem \ref{thm: kahler} also shows that the estimates for the orthogonal Laplacian and holomorphic Hessian in Theorem 1.1 (i) of \cite{ni2018comparison} can be improved. Since we are mainly interested in ordinary Laplacian comparison, we omit the details here.

Clearly we have the K\"ahler analogue of Theorem \ref{thm: bonnet myers} and its corollaries by the same argument. For simplicity which just states some particular results.
\begin{theorem}\label{thm: cpx myers}
Let $(M, g)$ be a K\"{a}hler manifold. Let $r_0$ be the smallest positive zero of $ s_{k_1} s_{k_2}$.
\begin{enumerate}
\item
If
\begin{align*}
\limsup_{r\to r_0 ^-} \left[
\int_{0}^{r}\left(\frac{s_{k_1}(t)^2}{s_{k_1}(r)^2} \widehat{\mathrm{Ric}}_{k_1}^\perp (\gamma_\theta'(t)) + \frac{s_{k_2}(t)^2}{s_{k_2}(r)^2} \widehat {\mathrm{H}}_{\frac{k_2}{2}}(\gamma_\theta'(t)) \right)dt
-\left((2n-2)\frac{s_{k_1}'(r)}{s_{k_1}(r)} + \frac{s_{k_2}'(r)}{s_{k_2}(r)}\right)\right]=\infty
\end{align*}
for any $\theta \in S_pM$,
then $d_p \le r_0$ on $ M$, $ M$ is compact and $ \pi_1(M)$ is finite.
\item
If $\displaystyle \int_{\mathcal B_g(r, p)} \left[\frac{s_{k_1}(t)^2}{s_{k_1}(r)^2} \widehat{\mathrm{Ric}}_{k_1}^\perp (\partial _t) +\frac{s_{k_2} (t)^2}{s_{k_2}(r)^2}  \widehat {\mathrm{H}}_{\frac{k_2}{2}}(\partial _t)\right] dV\ge 0$
for any $r\in (0, r_0)$,
then $\frac{|\mathcal B_g(r, p)|}{| B_{\overline g}(r)|}$ is non-increasing on $ (0, r_0]$. In particular, $|M|\le |B_{\overline g}(r_0)|$.

If $|\mathcal B_g(r, p)|=| B_{\overline g}(r)|$ ($ r\le r_0<\infty$), then $ \mathcal B_g(r, p)$ is isometric to $ B_{\overline g}(r)$.

\end{enumerate}
\end{theorem}
For example, if $\limsup_{r\to\frac{\pi}{2}^-}
\left[\int_{0}^{r} \frac{\sin(2t)^2}{\sin(2r)^2}\widehat {\mathrm{H}}_{2}(\gamma_\theta'(t)) dt- 2\cot(2r) \right]
=\infty$, then the diameter of $M$ is bounded above by $ \frac{\pi}{2}$.
If $\int_{0}^{r} \sin^2(t) \widehat{\mathrm{Ric}}_{1}^\perp (\gamma_\theta'(t)) dt \ge 0 $ and $\int_{0}^{r} \sin^ 2 (2t) \widehat {\mathrm{H}}_{2}(\gamma_\theta'(t)) dt\ge 0$
for any $\theta\in S_pM$ and $r\in (0, \frac{\pi}{2})$, then $|M|\le |\mathbb {CP}^n|$.

\subsection{Comparison theorems for distance from a complex submanifold}
As in the Riemannian case, to study the Laplacian of the distance function of a complex submanifold, it is natural to define the $\ell$-holomorphic sectional curvature as follows.

Let $W$ be a subspace of $ T_xM$ which is $ J$-invariant with $ \dim_{\mathbb R}(W)=2\mathrm{dim}_{\mathbb C}(W)=2\ell$, and $ v\in T_xM$.
Inspired by \cite{li2005comparison, ni2018comparison}, we define the $\ell$-holomorphic sectional curvature to be
\begin{align*}
{\mathrm{H}}^\ell (W, v)=\sum_{j=1}^{\ell} \left(\langle R(e_j, v)v, e_j\rangle +\langle R(Je_j, v)v, Je_j\rangle\right)
\end{align*}
where $e_1, Je_1, \cdots, e_\ell, J e_\ell$ is an orthonormal basis of $ W$. It is easy to check that this is well-defined.
Similar as before we define
\begin{align*}
\widehat {\mathrm{H}}^\ell_k(W, v)={\mathrm{H}}^\ell(W, v)-2\ell k g(v,v).
\end{align*}
When $\ell=n-1$ and $ W=\left(\mathrm{span}(v, Jv)\right)^\perp$ this is $ \widehat{\mathrm{Ric}}^\perp_k(v)$ and when $ \ell=1$ and $ W=\mathrm{span}(v, Jv)$ this is $ \widehat {\mathrm{H}}_k(v)$. We say $ \widehat {\mathrm{H}}^{\ell, \perp}_k\ge 0$ if $ \widehat{\mathrm{H}}^\ell_k (W, v)\ge 0$ for all such $ v$, $ W$ with $ \mathrm{span}(v, Jv)\perp W$.

The following result can be compared to \cite[Cor. 2.1]{tam2012some}
\begin{theorem}\label{thm: cpx submfd}
Suppose $\Sigma$ is a complex submanifold of a K\"ahler manifold $ M$ with $ \dim_{\mathbb C}(\Sigma)=\ell$. Let $ d_\Sigma$ be the distance from $ \Sigma$,
and $(r, \theta, z)$ be the Fermi coordinates of $ x$. Assume $ s_{k_i}$, $ c_{k_i}$ are positive on $ (0, r]$.

\begin{enumerate}
\item
We have
\begin{align*}
\Delta d_\Sigma (x)
\le (\log \overline F)'(r)-\psi(r, \theta, z)
\end{align*}
where $\overline F(r)={c_{k_1}(r)^{2\ell}s_{k_1}(r)^{2n-2\ell-2}s_{k_2}(r)}$,
\begin{align*}
\psi(r, \theta, z)
=\int_{0}^{r}
\left(\frac{c_{k_1}(t)^2}{c_{k_1}(r)^2}\widehat {\mathrm{H}}^\ell_{k_1} (P^t_\theta(T_z\Sigma), \partial_t)
+\frac{s_{k_1}(t)^2}{s_{k_1}(r)^2}\widehat {\mathrm{H}}^{n-\ell-1}_{k_1}\left(N_t, \partial_t\right)+ \frac{s_{k_2}(t)^2}{s_{k_2}(r)^2}\widehat {\mathrm{H}}_{\frac{k_2}{2}}(\partial_t)
\right)dt
\end{align*}
and $N_t=P^t_\theta\left(N_z \Sigma\cap \left(\mathrm{span}(\theta, J \theta)\right)^\perp\right)$.
\item
The volume element $F(r, \theta, z)$ of $ M$ satisfies
\begin{align*}
F(r, \theta, z)\le \exp(-\phi(r, \theta, z))\overline F(r)
\end{align*}
where
\begin{align*}
&\phi(r, \theta, z)\\
=&
\int_{0}^{r}\left[(\mathrm{tg}_{k_1}(r)-\mathrm{tg}_{k_1}(t))c_{k_1}(t)^2 \widehat {\mathrm{H}}^\ell_{k_1} (P^t_\theta(T_z\Sigma), \partial_t)
+(\mathrm{ct}_{k_1}(t)-\mathrm{ct}_{k_1}(r)) s_{k_1}(t)^2 \widehat {\mathrm{H}}^{n-\ell-1}_{k_1}\left(N_t, \partial_t\right)\right. \\
&\left. +(\mathrm{ct}_{k_2}(t)-\mathrm{ct}_{k_2}(r))s_{k_2}(t)^2 \widehat {\mathrm{H}}_{\frac{k_2}{2}}(\partial_t)
\right]dt.
\end{align*}
In particular, if $\widehat {\mathrm{H}}^\ell_{k_1}$, $\widehat {\mathrm{H}}^{n-\ell-1}_{k_1}$ and $ \widehat {\mathrm{H}}_{\frac{k_2}{2}}$ are non-negative, then $ F(r, \theta, z)\le \overline F(r)$.
\end{enumerate}
\end{theorem}

\begin{proof}
Since the proof is similar to Theorem \ref{thm: submfd}, we only indicate the changes here.
\begin{enumerate}
\item
Firstly, the parallel orthonormal frame along $\gamma_\theta$ is now
$\{F_1, \cdots, F_{2n}\}=\{E_1, JE_1, \cdots, E_{n}, JE_{n}=\gamma_\theta'\}$, with $F_1(0), \cdots, F_{2\ell}(0)\in T_z\Sigma $.
Note that a complex submanifold is minimal (cf. \cite[p. 171]{frankel1961manifolds}), so we define $X_i(t)=\frac{c_{k_1}(t)}{c_{k_1}(r)}F_i(t)$, $i=1, \cdots, 2\ell$. Similar to \eqref{ineq: hess l} we have
\begin{align*}
\sum_{j=1}^{2\ell}\nabla ^2 d_\Sigma(F_j(r), F_j(r))
\le&
\sum_{j=1}^{2\ell}I_{\Sigma}(X_j, X_j)
=-\int_{0}^{r}\frac{c_{k_1}(t)^2}{c_{k_1}(r)^2}\widehat {\mathrm{H}}^\ell_{k_1} (P^t_\theta(T_z\Sigma), \partial_t)dt+2\ell \frac{c_{k_1}'(r)}{c_{k_1}(r)}.
\end{align*}
Similar to \eqref{ineq: hess2},
\begin{align*}
&\sum_{j=2\ell+1}^{2n}\nabla ^2 d_{\Sigma}(F_j(r), F_j(r))\\
\le& -\int_{0}^{r}\left(\frac{s_{k_1}(t)^2}{s_{k_1}(r)^2}\widehat {\mathrm{H}}^{n-\ell-1}_{k_1}\left(N_t, \partial_t\right)+ \frac{s_{k_2}(t)^2}{s_{k_2}(r)^2}\widehat {\mathrm{H}}_{\frac{k_2}{2}}(\partial_t)\right)dt
+(2n-2\ell-2)\frac{s_{k_1}'(r)}{s_{k_1}(r)}+\frac{s_{k_2}'(r)}{s_{k_2}(r)}
\end{align*}
where $N_t=P^t_\theta\left(N_z \Sigma\cap \left(\mathrm{span}(v, Jv)\right)^\perp\right)$.
Combining the two inequalities gives the result.
\item
Similar to Theorem \ref{thm: est} \eqref{est2}, we have
\begin{align*}
&\log F(r, \theta, z)\\
\le& \log \overline F(r)\\
&-\int_{0}^{r}\left[(\mathrm{tg}_{k_1}(r)-\mathrm{tg}_{k_1}(t))c_{k_1}(t)^2 \widehat {\mathrm{H}}^\ell_{k_1} (P^t_\theta(T_z\Sigma), \partial_t)
+(\mathrm{ct}_{k_1}(t)-\mathrm{ct}_{k_1}(r)) s_{k_1}(t)^2 \widehat {\mathrm{H}}^{n-\ell-1}_{k_1}\left(N_t, \partial_t\right)\right. \\
&\left. +(\mathrm{ct}_{k_2}(t)-\mathrm{ct}_{k_2}(r))s_{k_2}(t)^2 \widehat {\mathrm{H}}_{\frac{k_2}{2}}(\partial_t)
\right]dt.
\end{align*}
Here we have also used Lemma \ref{lem: aux}. This implies the result.
\end{enumerate}
\end{proof}

We have the following analogue of Theorem \ref{thm: submfd vol}.

\begin{theorem}\label{thm: cpx submfd vol}
Suppose $\Sigma$ is a complex submanifold of a K\"ahler manifold $ M$ with $ \dim_{\mathbb C} \Sigma=\ell$. Assume $k_i\in \mathbb R$, $ s_{k_i}, c_{k_i}>0$ on $ (0, r]$, $ \widehat{\mathrm{H}}_{k_1}^{j, \perp}\ge 0$ for $ j= \ell, n-\ell-1 $ and $ \widehat {\mathrm{H}}_{\frac{k_2}{2}}\ge 0$.
\begin{enumerate}
\item
We have
\begin{align*}
|S(r, \Sigma)|
\le |\mathbb S^{2n-2\ell-1}|\overline F(r) \int_\Sigma \fint_{S(N_z\Sigma)}\exp(-\phi) d\theta dS
\le |\mathbb S^{2n-2\ell-1}|\overline F(r) |\Sigma|
\end{align*}
where $\overline F$ and $ \phi\ge 0$ are given in Theorem \ref{thm: cpx submfd}.
\item
We have
\begin{align*}
\frac{d}{dr}\left(\frac{|S(r, \Sigma)|}{\overline A(r)}\right)
\le -\frac{1}{\overline A(r)} \int_ {S(r, \Sigma)} \phi\, dS
\le 0,
\end{align*}
where $\overline A(r)= |\Sigma||\mathbb S^{2n-2\ell-1}| \overline F(r)$ and $ \phi(x):=\phi(r(x), v(x), z(x))$.

Let $\overline V(r)= \int_{0}^{r}\overline A(u)du$, then we also have
\begin{align*}
\frac{d}{dr}
\left(\frac{|B(r, \Sigma)|}{\overline V(r)}\right)
\le-\frac{\overline A(r)}{\overline V(r)^2} \int_{0}^{r} \frac{\overline V(u)}{\overline A(u)} \int_{S(u, \Sigma)}\phi\, dS \, du\le 0.
\end{align*}

\end{enumerate}
\end{theorem}
\begin{remark}
It would also be interesting to look for quantitative comparison results for other types of special manifolds, such as quaternionic K\"ahler manifolds. For simplicity we will not do it here. We notice that there are sharp comparison results for quaternionic K\"ahler manifolds assuming a scalar curvature lower bound, as studied by Kong, Li and Zhou \cite{kong2008spectrum} with methods similar to \cite{li2005comparison}.
\end{remark}

\section{G\"unther-type theorems}\label{sec: gunther}
\subsection{Riemannian case}
We now give a generalization of G\"unther's theorem (\cite[Thm. 3.17]{gray2012tubes}), which gives a lower bound for the volume of the geodesic ball under a curvature upper bound. It is possible to work with a variable $k$ but for simplicity we assume $ k$ is constant in this section.

We define $B= \frac 12 g\odot g$, where $\odot$ is the Kulkarni-Nomizu product (\cite[p.47]{besse2008einstein}). More explicitly,
\begin{equation*}
\begin{split}
B(X, Y, Z, W)
=&\langle X, W\rangle\langle Y, Z\rangle- \langle X, Z\rangle\langle Y, W\rangle.
\end{split}
\end{equation*}
It is easy to see that $kB$ is the Riemann curvature tensor of a space form with curvature $ k$. Also define on $ (M, g)$ the $ 4$-tensor
\begin{align*}
\widehat{R}_k(X, Y, Z, W):=\langle R(X, Y)Z, W\rangle - kB(X, Y, Z, W).
\end{align*}
Clearly the curvature of $(M, g)$ is bounded above by $ k$ if and only if $ \widehat R_k(X, Y, Y, X)\le 0$. We will denote the metric of the simply connected space form $\overline M_k$ of curvature $ k$ by $ \overline g$.

\begin{theorem}\label{thm: gunther1}
Let $ g=dt^2+\beta_{ij}(t, \theta)d\theta^i d\theta^j$ in geodesic polar coordinates.
Let $x=(r, \theta)$ in geodesic polar coordinates centered at $p$.
Assume $ s_{k}>0$ on $ (0, r]$.
\begin{enumerate}
\item
We have
\begin{align*}
\Delta d_p(x)
\ge (n-1)\frac{s_k'(r)}{s_k(r)}-\sum_{i,j=1}^{n-1} \beta^{ij}(r, \theta) \int_{0}^{r}\widehat R_k\left(\partial_{\theta^i}, \partial_t, \partial_t, \partial_{\theta^j}\right) dt,
\end{align*}
where $ (\beta^{ij})=(\beta_{ij})^{-1}$.

\item
If $dV=F(r, \theta)dr d\theta$, then
\begin{equation*}
\begin{split}
F(r, \theta)\ge&
\exp \left[-\int_{0}^{r}\int_{0}^{\rho} \sum_{i,j=1}^{n-1}\beta^{ij}(\rho, \theta)\widehat{R}_k \left(\partial_{\theta^i}, \partial_t, \partial_t, \partial_{\theta^j}\right)dt \, d\rho \right] \overline F(r)
\end{split}
\end{equation*}
where $\overline F(r)=s_k(r)^{n-1}$.
\end{enumerate}
\end{theorem}

\begin{proof}
We use the same notation as in the proof of Theorem \ref{thm: est}.
Let $\overline \gamma$ be a geodesic segment of length $ r$ on $\overline M_k$ and $ \{\overline E_i\}_{i=1}^n$ be the parallel translation of a positive orthonormal basis along $ \overline \gamma$ such that $ \overline E_n=\overline \gamma'$. Suppose $Y_i(t)=\sum_{j=1}^{n-1}y_i^j (t)E_j(t)$. Then we define
$\overline Y_i (t)=\sum_{j=1}^{n-1} y_i^j (t)\overline E_j(t)$
and
$\overline X_i (t)=\sum_{j=1}^{n-1}\frac{s_k(t)}{s_k(r)}\overline E_i(t)$.
Note that $\overline Y_i (t)=\overline X_i (t)$ when $ t=0$ and $ r$. Denote the curvature of $\overline g$ by $ \overline R$, we have
\begin{equation}\label{eq: I2}
\begin{split}
I(Y_i, Y_i)
=&\int_{0}^{r}\left(\langle {Y_i}', {Y_i}'\rangle - R (Y_i, \gamma', \gamma', Y_i)\right) dt\\
=&\int_{0}^{r}\left(\langle {Y_i}', {Y_i}'\rangle - kB(Y_i, \gamma', \gamma', Y_i)-\widehat{R}_k (Y_i, \gamma', \gamma', Y_i)\right) dt\\
=&\int_{0}^{r}\left(\langle \overline {Y_i}', \overline {Y_i}'\rangle - \langle \overline R (\overline Y_i, \overline \gamma') \overline \gamma', \overline Y_i\rangle -\widehat{R}_k (Y_i, \gamma', \gamma', Y_i)\right) dt\\
=&I(\overline {Y_i}, \overline {Y_i})-\int_{0}^{r} \widehat{R}_k (Y_i, \gamma', \gamma', Y_i) dt.
\end{split}
\end{equation}
Notice that in geodesic polar coordinates,
$\frac{\partial}{\partial \theta^i}$ is a Jacobi field and
\begin{align}\label{eq: beta}
\sum_{i, j=1}^{n-1}\beta^{ij}(r, \theta)\widehat R_k\left(\left. \frac{\partial}{\partial \theta^i} \right|_{(t, \theta)}, \gamma'(t), \gamma'(t), \left. \frac{\partial}{\partial \theta^j}\right|_{(t, \theta)}\right)=\sum_{i=1}^{n-1} \widehat R_k(Y_i(t), \gamma'(t), \gamma'(t), Y_i(t))
\end{align}
is independent of the choice of spherical coordinates and $Y_i$ (recall $Y_i=Y_i^{r, \theta}$).
So using \eqref{eq: I2}, index lemma and \eqref{eq: log F} applied to $\overline X_i$, we have
\begin{equation}\label{eq: I3}
\begin{split}
\Delta d_p(x)=(\log F)'(r, \theta)
=&\sum_{i=1}^{n-1}I(Y_i, Y_i)\\
=&\sum_{i=1}^{n-1} \left(I(\overline Y_i, \overline Y_i)- \int_{0}^{r}\widehat R_k(Y_i, \gamma', \gamma', Y_i)dt\right)\\
\ge&\sum_{i=1}^{n-1} \left(I(\overline X_i, \overline X_i)- \int_{0}^{r}\widehat R_k(Y_i, \gamma', \gamma', Y_i) dt\right)\\
=& (n-1)\frac{s_k'(r)}{s_k(r)}- \sum_{i, j=1}^{n-1} \beta^{ij}(r, \theta) \int_{0}^{r}\widehat R_k\left(\partial_ {\theta^i}, \partial_t, \partial_t, \partial_{\theta^j}\right) dt.
\end{split}
\end{equation}
The lower bound for $F$ is straightforward.

\end{proof}

Notice the similarity with \eqref{ineq: lap d1} and that the quantity $\sum_{i, j}\beta^{ij}(r, \theta)\widehat R_k\left(\partial_ {\theta^i}, \partial_t, \partial_t, \partial_{\theta^j}\right)$ is independent of the choice of spherical coordinates. Although a volume lower bound is not guaranteed by an upper bound of the Ricci curvature, this quantity plays a role similar to $ \widehat {\mathrm{Ric}}_k$ in this case. Furthermore, we can upper bound $\sum_{i,j}\beta^{ij}(r, \theta)\widehat R_k(\partial _{\theta^i}, \partial _t, \partial _t, \partial _{\theta^j})$ if we further assume $\ell\le \mathrm{Rm}\le k$, by the classical Hessian comparison.

Similar to Theorem \ref{thm: area vol est}, we have the following result, which can be regarded as the relative version of the G\"unther's inequalilty.
\begin{theorem}\label{thm: gunther area vol est}
Suppose $r<\mathrm{inj}(p)$ and $ s_k>0$ on $ (0, r]$, then
\begin{equation*}
\frac{d}{dr} \left(\frac{|S_g(r, p)|}{|S_{\overline g}(r)|}\right)
\ge -\frac{1}{|S_{\overline g}(r)|} \int_{B_g(r, p)}\sum_{i, j=1}^{n-1}\beta^{ij}(r, \theta)\widehat{R}_k \left(\partial_{\theta^i}, \partial_t, \partial_t, \partial_{\theta^j}\right) dV
\end{equation*}
and
\begin{align*}
\frac{d}{dr}\left(\frac{|B_g(r, p)|} {|B_{\overline g}(r)|} \right)
\ge -\frac{\overline F(r)}{|B_{\overline g}(r)|^2} \int_{0}^{r} \frac{|B_{\overline g}(u)|}{\overline F(u)} \int_{B_g(u, p)}\sum_{i, j=1}^{n-1}\beta^{ij}(u, \theta)\widehat{R}_k \left(\partial_{\theta^i}, \partial_t, \partial_t, \partial_{\theta^j}\right) dV\, du.
\end{align*}

In both cases, the equality holds if and only if $B_g(r, p)$ is isometric to $ B_{\overline g}(r)$.
\end{theorem}

As an analogue of Proposition \ref{prop: isop}, we have
\begin{proposition}
If $\int_{B_g(\rho, p)}\sum_{i, j=1}^{n-1}\beta^{ij}(\rho, \theta)\widehat{R}_k \left(\partial_{\theta^i}, \partial_t, \partial_t, \partial_{\theta^j}\right) dV\le 0$ for all $\rho\in [0, r]$, then $\frac{|B_g(t, p)|}{|S_g(t, p)|} \le \frac{|B_{\overline g}(t)|}{|S_{\overline g}(t)|} $ for $t\in [0, r]$.
\end{proposition}

\begin{theorem}\label{thm: gunther submfd}
Suppose $\Sigma$ is an $ \ell$-dimensional submanifold of $ (M, g)$. Let $ d_\Sigma: M\to \mathbb R$ be the distance from $ \Sigma$, $ g=dr^2+ \sum_{i, j=1}^{n-\ell-1}\beta_{ij}(r, \theta, z)d\theta^i d\theta^j+ \sum_{i, j=1}^\ell \alpha_{ij}(r, \theta, z)dz^i dz^j$ in Fermi coordinates w.r.t. $ \Sigma$.
Let $x=(r, \theta, z)$ in Fermi coordinates.
Assume $ s_k>0$ on $ (0, r]$ and that the first zero of $ t\mapsto c_k(t)+\lambda s_k (t)$ (if exists) appears no earlier than the cut distance in the direction $ \theta$, where $ \lambda=\min_{v\in S_p\Sigma} A_\theta(v, v)$.

\begin{enumerate}
\item\label{item: gunther submfd1}
We have
\begin{align*}
\Delta d_\Sigma(x)
\ge&
(\log \overline F)'(r, \theta, z)-\phi(r, \theta, z)
\end{align*}
where $\overline F(r, \theta, z) =s_k(r)^{n-\ell-1} \det \left[ c_k(r)\mathrm{Id}+s_k(r)A_\theta \right]$,
$ (\alpha^{ij})=(\alpha_{ij})^{-1}$ and
\begin{align*}
\phi(r, \theta, z)=
\sum_{i, j=1}^{n-\ell-1}\beta^{ij}(r, \theta, z) \int_{0}^{r} \widehat R_k(\partial_{\theta^i}, \partial_t, \partial_t, \partial_{\theta^j})dt+\sum_{i, j=1}^{\ell}\alpha^{ij}(r, \theta, z)\int_{0}^{r}\widehat R_k\left({\partial_{z^i}}, \partial_t, \partial_t, {\partial_{z^j}}\right) dt.
\end{align*}
Here we regard $A_\theta$ as a $(1, 1)$-tensor.
\item\label{item: gunther submfd2}
If $dV=F(r, \theta, z)dr \,d\theta\, dz$, then
$F(r, \theta, z)\ge \overline F(r, \theta, z)\exp\left[ -\int_{0}^{r}\phi(\rho, \theta, z)d\rho\right]$.
\item\label{item: gunther submfd3}
Let $\overline A(r)=\int_\Sigma \int_{S(N_z\Sigma)} \overline F(r, \theta, z)d\theta \,dz$ and $\overline V(r)=\int_{0}^{r}\overline A (t)dt$.
Suppose $ \overline F '(\rho, \theta, z)\ge 0 $ and $\phi(\rho, \theta, z)\le 0$ for all $\rho\in(0, r_0)$ and $(\theta, z)\in S(N\Sigma)$.
Then on $[0, r_1]$, $|S_g(r, \Sigma)|-\overline A(r)$ is non-negative and non-decreasing, and $|B_g(r, \Sigma)|-\overline V(r)$ is non-negative, non-decreasing and convex. Here $r_1=\min\{r_0, \mathrm{inj}(\Sigma) \}$.
\end{enumerate}

\end{theorem}

\begin{proof}
We use the notation in the proof of Theorem \ref{thm: submfd}. Furthermore assume $\{E_i\}$ diagonalizes $ A_\theta$ with eigenvalues $ \{\lambda_i\}$.
Choose $\overline \Sigma$ to be a (local) $ \ell$-dimensional submanifold in $ \overline M$ such that at $ \overline p\in \overline \Sigma$ and for $ \overline \theta\in S(N_{\overline p}\overline \Sigma)$, the second fundamental form $ \overline A_ {\overline \theta} $ agrees with $ A_ \theta $. So there exists an orthonormal basis $ \overline E_i$ of $ T_{\overline p}\overline \Sigma$ such that $ \overline A_{\overline \theta}(\overline E_i, \overline E_j) =\lambda_i \delta_{ij}$. As before, parallel transport $\overline E_i$ along $ \overline \gamma=\overline \gamma_{\overline \theta}$. For $ Y_i(t)=\sum_{j=1}^\ell y_i^j (t)E_j(t)$, define $ \overline Y_i(t)=\sum_{j=1}^\ell y_i^j (t)\overline E_j(t)$ and $ \overline X_i(t)= \frac{c_k(t)+\lambda_i s_k(t)}{c_k(r)+\lambda_i s_k(r)}\overline E_i(t)$. Note that $ \overline X_i$ are adapted to $ \overline \Sigma$ along $ \overline \gamma$. So by index lemma and \eqref{eq: I2},
\begin{equation*}
\begin{split}
I_\Sigma(Y_i, Y_i)
=&\int_{0}^{r}\left(\langle {Y_i}', {Y_i}'\rangle - R (Y_i, \gamma', \gamma', Y_i)\right) dt+A_\theta(Y_i(0), Y_i(0)) \\
=&I_{\overline \Sigma}(\overline {Y_i}, \overline {Y_i})-\int_{0}^{r} \widehat{R}_k (Y_i, \gamma', \gamma', Y_i) dt+A_\theta(Y_i(0), Y_i(0)) - \overline A_{\overline \theta}(\overline Y_i(0), \overline Y_i(0))\\
\ge& I_{\overline \Sigma}(\overline X_i, \overline X_i)- \int_{0}^{r}\widehat R_k(Y_i, \gamma', \gamma', Y_i) dt \\
=& \frac{c_k'(r)+\lambda_i s_k'(r)}{c_k(r)+\lambda_i s_k(r)}- \int_{0}^{r}\widehat R_k(Y_i, \gamma', \gamma', Y_i) dt.
\end{split}
\end{equation*}
In Fermi coordinates, $\{\frac{\partial}{\partial z^i}\}_{i=1}^\ell $ are Jacobi fields adapted to $ \Sigma$ (\cite[Lem. 2.9]{gray2012tubes}). So as in Theorem \ref{thm: gunther1}, we have
$\sum_{i=1}^\ell I_\Sigma(Y_i, Y_i)
\ge \sum_{i=1}^\ell\frac{c_k'(r)+\lambda_i s_k'(r)}{c_k(r)+\lambda_i s_k(r)}- \sum_{i, j=1}^{\ell}\alpha^{ij}(r, \theta, z)\int_{0}^{r}\widehat R_k\left({\partial_{z^i}}, \partial_t, \partial_t, {\partial_{z^j}}\right) dt$.

Using \eqref{eq: I3} and \eqref{eq: two hess}, we can then proceed as in Theorem \ref{thm: submfd} to show that
\begin{align*}
\Delta d_\Sigma(x)
\ge&
(n-\ell-1)\frac{s_k'(r)}{s_k(r)}
+\sum_{i=1}^\ell\frac{c_k'(r)+\lambda_i s_k'(r)}{c_k(r)+\lambda_i s_k(r)}\\
&-\sum_{i, j=1}^{n-\ell-1}\beta^{ij}(r, \theta, z) \int_{0}^{r} \widehat R_k(\partial_{\theta^i}, \partial_t, \partial_t, \partial_{\theta^j})dt
- \sum_{i, j=1}^{\ell}\alpha^{ij}(r, \theta, z)\int_{0}^{r}\widehat R_k\left({\partial_{z^i}}, \partial_t, \partial_t, {\partial_{z^j}}\right) dt.
\end{align*}
Note that
$\overline F(r, \theta, z)
=s_k(r)^{n-\ell-1} \prod_{i=1}^{\ell} (c_k(r)+\lambda_i s_k(r))
=s_k(r)^{n-\ell-1}\det \left[ c_k(r)\mathrm{Id}+s_k(r)A_\theta \right]
$.
\eqref{item: gunther submfd1} and \eqref{item: gunther submfd2} follow. \eqref{item: gunther submfd3} is similar to Theorem \ref{thm: submfd vol} \eqref{item: submfd3}.
\end{proof}
\subsection{K\"ahler case}
On a K\"ahler manifold $(M, g, J)$, define the $4$-tensor
\begin{align*}
&C(X, Y, Z, W)\\
=&\frac{1}{2} \left[ \langle X, W\rangle \langle Y, Z\rangle - \langle X, Z\rangle \langle Y, W\rangle
+\langle X, JW\rangle \langle Y, JZ\rangle- \langle X, JZ\rangle \langle Y, JW\rangle +2\langle X, JY\rangle \langle W, JZ\rangle
\right].
\end{align*}
Let $k\in \mathbb R$.
Then $kC$ is the Riemann curvature tensor of a complex space form with holomorphic sectional curvature $k$ (\cite[Prop. 7.2]{kobayashi1969foundations}). Also define on $ M$ the $ 4$-tensor (which is different from $\widehat R_k$)
\begin{align*}
\widehat {\mathrm{R}}_k(X, Y, Z, W):=\langle R(X, Y)Z, W\rangle -k C(X, Y, Z, W).
\end{align*}

\begin{theorem}\label{thm: cpx gunther1}
Let $(M,g)$ be a K\"ahler manifold and $ g=dt^2+\beta_{ij}(t, \theta)d\theta^i d\theta^j$ in geodesic polar coordinates centered at $p$.
Let $x=(r, \theta)$ in geodesic polar coordinates.
Assume there is no cut point of $p$ along $ \gamma_\theta$ on $ [0, r]$ and $ s_{4k}>0$ on $ (0, r]$.
\begin{enumerate}
\item
We have
\begin{equation*}
\begin{split}
\Delta d_p(x)
\ge&
(2n-2)\frac{s_{k}'(r)}{s_{k}(r)} + \frac{s_{4k}'(r)}{s_{4k}(r)}- \sum_{i, j=1}^{2n-1} \beta^{ij}(r, \theta) \int_{0}^{r}\widehat {\mathrm{R}}_k\left(\partial_ {\theta^i}, \partial_t, \partial_t, \partial_{\theta^j}\right) dt.
\end{split}
\end{equation*}

\item
If $dV=F(r, \theta)dr d\theta$, then
\begin{equation*}
\begin{split}
F(r, \theta)\ge&
\exp \left[-\int_{0}^{r}\int_{0}^{\rho} \sum_{i,j=1}^{2n-1} \beta^{ij}(\rho)\widehat{\mathrm{R}}_k \left(\partial_{\theta^i}, \partial_t, \partial_t, \partial_{\theta^j}\right)dt \, d\rho \right] \overline F(r)
\end{split}
\end{equation*}
where $\overline F(r)=s_k(r)^{2n-2}s_{4k}(r)$.
\end{enumerate}
\end{theorem}

\begin{proof}

We use the same notation as in the proof of Theorem \ref{thm: est}.
Let $\{F_i\}_{i=1}^{2n}$ be defined as in the proof of Theorem \ref{thm: kahler}. Let $\overline \gamma$ be a geodesic segment of length $ r$ on $\overline M_k$ and $\{\overline F_i\}_{i=1}^{2n}$ be defined analogously on $\overline \gamma$.
Suppose $Y_i(t)=\sum_{j=1}^{2n-1}y_i^j (t)F_j(t)$, $i=1, \cdots, 2n-1$. Then we define
$\overline Y_i (t)=\sum_{j=1}^{2n-1} y_i^j (t)\overline F_j(t)$
and
$\overline X_i (t)=\sum_{j=1}^{2n-1}\frac{s_{k}(t)}{s_{k}(r)}\overline F_i(t)$ for $i=1, \cdots, 2n-2$ and $\overline X_{2n-1}(t)=\frac{s_{4k}(t)}{s_{4k}(r)}\overline F_{2n-1}(t)$.
Note that $\overline Y_i (t)=\overline X_i (t)$ when $t=0$ and $ r$.
As in \eqref{eq: I2}, we have
\begin{equation}\label{eq: I2 cpx}
\begin{split}
I(Y_i, Y_i)
=&I(\overline {Y_i}, \overline {Y_i})-\int_{0}^{r} \widehat{\mathrm R}_k (Y_i, \gamma', \gamma', Y_i) dt.
\end{split}
\end{equation}
As in \eqref{eq: beta},
$\sum_{i, j=1}^{2n-1}\beta^{ij}(r, \theta)\widehat {\mathrm{R}}_k\left(\frac{\partial}{\partial \theta^i} , \gamma'(t), \gamma'(t), \frac{\partial}{\partial \theta^j}\right)=\sum_{i=1}^{2n-1} \widehat {\mathrm{R}}_k(Y_i(t), \gamma'(t), \gamma'(t), Y_i(t))
$.
So using \eqref{eq: I2 cpx}, index lemma and \eqref{eq: log F}, similar to Theorem \ref{thm: gunther1}, we have
\begin{equation*}\label{eq: I3 cpx}
\begin{split}
\Delta d_p(x)=(\log F)'(r, \theta)
\ge& (2n-2)\frac{s_{k}'(r)}{s_{k}(r)} + \frac{s_{4k}'(r)}{s_{4k}(r)}- \sum_{i, j=1}^{2n-1} \beta^{ij}(r, \theta) \int_{0}^{r}\widehat {\mathrm{R}}_k\left(\partial_ {\theta^i}, \partial_t, \partial_t, \partial_{\theta^j}\right) dt.
\end{split}
\end{equation*}
The lower bound for $F$ is straightforward.

\end{proof}

The K\"ahler analogue of Theorem \ref{thm: gunther area vol est} is the following
\begin{theorem}\label{thm: cpx gunther area vol est}
With the same assumptions in Theorem \ref{thm: cpx gunther1},
\begin{equation*}
\frac{d}{dr} \left(\frac{|S_g(r, p)|}{|S_{\overline g}(r)|}\right)
\ge -\frac{1}{|S_{\overline g}(r)|} \int_{B_g(r, p)}\sum_{i, j=1}^{2n-1}\beta^{ij}(r, \theta)\widehat{\mathrm{R}}_k \left(\partial_{\theta^i}, \partial_t, \partial_t, \partial_{\theta^j}\right) dV
\end{equation*}
and
\begin{align*}
\frac{d}{dr}\left(\frac{|B_g(r, p)|} {|B_{\overline g}(r)|} \right)
\ge -\frac{\overline F(r)}{|B_{\overline g}(r)|^2} \int_{0}^{r} \frac{|B_{\overline g}(u)|}{\overline F(u)} \int_{B_g(u, p)}\sum_{i, j=1}^{2n-1}\beta^{ij}(u, \theta)\widehat{\mathrm{R}}_k \left(\partial_{\theta^i}, \partial_t, \partial_t, \partial_{\theta^j}\right) dV\, du.
\end{align*}
In particular, if
$\displaystyle \int_{B_g(\rho, p)}\sum_{i, j=1}^{2n-1} \beta^{ij}(\rho, \theta)\widehat{\mathrm{R}}_k \left(\partial_{\theta^i}, \partial_t, \partial_t, \partial_{\theta^j}\right) dV\le 0$
for all $\rho\in (0, r)$,
then
$\frac{| B_g(\rho, p)|}{| B_{\overline g}(\rho)|}$ is non-decreasing
on $(0, r)$.

\end{theorem}

\begin{theorem}\label{thm: cpx gunther submfd}
Suppose $\Sigma$ is a complex submanifold of a K\"ahler manifold $ M$ with $ \dim_{\mathbb C}(\Sigma)=\ell$. Let $ d_\Sigma$ be the distance from $ \Sigma$, $ g=dr^2+ \sum_{i, j=1}^{2n-2\ell-1}\beta_{ij}(r, \theta, z)d\theta^i d\theta^j+ \sum_{i, j=1}^{2\ell} \alpha_{ij}(r, \theta, z)dz^i dz^j$ in Fermi coordinates w.r.t. $ \Sigma$.
Let $x=(r, \theta, z)$ in Fermi coordinates.
Assume $ s_{4k}>0$ on $ (0, r]$ and that the first zero of $ t\mapsto c_k(t)+\lambda s_k (t)$ (if exists) appears no earlier than the cut distance in the direction $ \theta$, where $ \lambda=\min_{v\in S_p\Sigma} A_\theta(v, v)$.

\begin{enumerate}
\item
We have
\begin{align*}
\Delta d_\Sigma(x)
\ge&
(\log \overline F)'(r, \theta, z)-\phi(r, \theta, z)
\end{align*}
where $\overline F(r, \theta, z) =s_k(r)^{2n-2\ell-2}s_{4k}(r) \det \left[ c_k(r)\mathrm{Id}+s_k(r)A_\theta \right]$,
$ (\alpha^{ij})=(\alpha_{ij})^{-1}$ and
\begin{align*}
\phi(r, \theta, z)
=\sum_{i, j=1}^{2n-2\ell-1}\beta^{ij}(r, \theta, z) \int_{0}^{r} \widehat {\mathrm{R}}_k(\partial_{\theta^i}, \partial_t, \partial_t, \partial_{\theta^j})dt
- \sum_{i, j=1}^{2\ell}\alpha^{ij}(r, \theta, z)\int_{0}^{r}\widehat {\mathrm{R}}_k\left({\partial_{z^i}}, \partial_t, \partial_t, {\partial_{z^j}}\right) dt.
\end{align*}
Here we regard $A_\theta$ as a $(1, 1)$-tensor.
\item
Let $dV=F(r, \theta, z)dr d\theta dz$, then we have
$F(r, \theta, z)\ge \overline F(r, \theta, z)\exp\left[ -\int_{0}^{r}\phi(\rho, \theta, z)d\rho\right]$.
\end{enumerate}

\end{theorem}

\begin{proof}
Since the proof is similar to Theorem \ref{thm: gunther submfd}, we just outline it here. We use the notation in the proof of Theorem \ref{thm: cpx submfd}.
Furthermore assume $\{F_i\}_{i=1}^{2\ell}$ diagonalizes $ A_\theta$.
Choose $\overline \Sigma$ to be a (local) $ \ell$-dimensional complex submanifold in $ \overline M_k$ such that at $ \overline p\in \overline \Sigma$ and for $ \overline \theta\in S(N_{\overline p}\overline \Sigma)$, the second fundamental form $ \overline A_ {\overline \theta} $ agrees with $ A_ \theta $. Let $\{\overline F_i\}$ be analogously defined along $\overline \gamma=\overline \gamma_{\overline \theta}$.

For $ Y_i(t)=\sum_{j=1}^{2\ell} y_i^j (t)F_j(t)$, define $ \overline Y_i(t)=\sum_{j=1}^{2\ell} y_i^j (t)\overline F_j(t)$ and $ \overline X_i(t)= \frac{c_k(t)+\lambda_i s_k(t)}{c_k(r)+\lambda_i s_k(r)}\overline F_i(t)$. We can then proceed as in Theorem \ref{thm: gunther submfd} to obtain the result.

\end{proof}

\section{Some applications}\label{sec: applications}

Theorem \ref{thm: est}, Theorem \ref{thm: area vol est}, or Proposition \ref{prop: radial} can be used to provide weaker assumptions to many classical theorems, and at the same time give better estimates (if desired). For example, if only an integral version of the Laplacian comparison theorem for a radial function is used to prove a certain result, Proposition \ref{prop: radial} can often be a substitute to provide weaker weaker assumption than a pointwise Ricci curvature lower bound or even an integral bound along all geodesics emanating from a point. We illustrate some of the possibilities here.

The following result characterizes the equality case in Theorem \ref {thm: bonnet myers} and generalizes Cheng's maximal diameter theorem.

\begin{theorem}\label{thm: diameter}
Let $M$ be a complete Riemannian manifold.
Assume
\begin{enumerate}
\item
$s_k(t)=s_k(r_0-t)$ where $ r_0$ is the first positive zero of $ s_k$. \item
There exists $p_1, p_2\in M$ such that $ d(p_1, p_2)=r_0$.
\item
For any $\theta\in S_{p_i}M$ and any $ r \in(0, r_0)$, we have
$\displaystyle \int_{0}^{r}\widehat {\mathrm{Ric}}_{k} \left(s_k(t) \gamma_\theta'(t) \right) dt\ge 0$.
\end{enumerate}
Then $M$ is isometric to $ (\overline M=[0, r_0]\times \mathbb S^{n-1}, \overline g=dt^2+s_k(t)^2g_{\mathbb S^{n-1}})$.
\end{theorem}

\begin{proof}
By Theorem \ref{thm: bonnet myers} and Theorem \ref{thm: area vol est},
for any $r<r_0$, we have
$\frac{|\mathcal B_g(r, p_1)|}{|B_{\overline g}(r)|}
\ge\frac{|\mathcal B_{g} (r_0, p_1) |}{|B_{\overline g}(r_0)|}
=\frac{|M|}{|\overline M|}$.
So
$|\mathcal B_g(r, p_1)| \ge \frac{|M|}{|\overline M|} |B_{\overline g}(r)|$ and
by the same reason, $|\mathcal B_{g} (r_0-r, p_2) | \ge \frac{|M|}{|\overline M|} |B_{\overline g}(r_0-r)|$.
Adding them gives
\begin{align*}
|\mathcal B_g(r, p_1)| +|\mathcal B_{g} (r_0-r, p_2) | \ge \frac{|M|}{|\overline M|} (|B_{\overline g}(r)|+|B_{\overline g}(r_0-r)|)=|M|.
\end{align*}
But it is easy to see that $\mathcal B_g(r, p_1)$ and $ \mathcal B_g (r_0-r, p_2) $ must be disjoint, and so the inequalities above are all equalities. In particular,
\begin{align*}
\frac{|\mathcal B_g(r, p_1)|}{|B_{\overline g}(r)|}
=\frac{|M|}{|\overline M|}
\end{align*}
is constant for all $0<r\le r_0$. From this it follows from the proof of Theorem \ref{thm: area vol est} that $M=\overline {\mathcal B_{r_0}(p_1)}$ is isometric to the closed ball $ \overline {B_{\overline g}(r_0)}$ with metric $ dr^2+s_k(r)^2g_{\mathbb S^{n-1}}$, which is $ \overline M$.
\end{proof}

The estimate \eqref{ineq: lap d1} can be used to weaken the assumptions in Cheng's eigenvalue comparison theorem (\cite[Theorem 1.1]{cheng1975eigenvalue}). For constant $k$, we denote the geodesic ball of radius $r$ in the simply connected space form $ \overline M_k$ of curvature $ k$ by $ B_{\overline g}(r)$.

\begin{theorem}\label{thm: Cheng first thm}
Let $k$ be constant.
Suppose $M$ is a Riemannian manifold and $ p\in M $ such that
$\displaystyle \int_{\mathcal B_g'(\rho, p)} \widehat{\mathrm{Ric}}_k \left(s_k(t) \partial_t\right) d V \ge 0$ for all $ 0\le \rho\le r<\mathrm{diam}(M_k)$.
Then
$\lambda_1(\mathcal B_g(r, p)) \le\lambda_1(B_{\overline g}(r))$,
where $\lambda_1$ is the first eigenvalue with respect to the Dirichlet boundary condition.
The equality holds if and only if $\mathcal B_g(r, p)$ is isometric to $ B_{\overline g}(r)$.
\end{theorem}

\begin{proof}
Let $\lambda=\lambda_1(B_{\overline g} (r))$ and $ \phi>0$ be the first eigenfunction on $ \overline { B_{\overline g}(r)}$, which is radial (cf. \cite{cheng1975eigenvalue} p. 290) and $ \frac{d\phi}{dt}<0$ on $ (0, r)$ (\cite[Lemma 3.7]{cheng1975eigenfunctions}). Consider $ \phi\circ d_p: \overline {\mathcal B_g(r, p)}\to \mathbb R$ as a test function, simply written as $ \phi$. Then Proposition \ref{prop: radial} gives
\begin{align*}
\int_{\mathcal B_g(r, p)} |\nabla \phi|^2 \le \lambda \int_{\mathcal B_g(r, p)} \phi^2.
\end{align*}
By the minimization property of the first eigenvalue, the result follows.
The equality case is the same as \cite[Theorem 1.1]{cheng1975eigenvalue}.
\end{proof}

Since \cite[Theorem 1.1]{cheng1975eigenvalue} is used to prove \cite[Theorem 2.1]{cheng1975eigenvalue}, by analyzing the proof of \cite[Theorem 2.1]{cheng1975eigenvalue} and using Theorem \ref{thm: Cheng first thm} we also have
\begin{theorem}\label{thm: Cheng second thm}
Let $k$ be constant.
Suppose $M$ is an $ n$-dimensional compact Riemannian manifold with $ \mathrm{diam}(M)=d_M$. Suppose $ \frac{d_M}{2}<\mathrm{diam}(\overline M_k)$ for some $ k$ and for all $ p\in M$ and $ t\in [0, \frac{1}{2}d_M]$, we have
$\int_{\mathcal B_g'(t, p)} \widehat{\mathrm{Ric}}_k \left(s_k(\rho) \partial _\rho\right) d V \ge 0$.
Then
$\mu_i(M) \le \lambda_1\left(B_{\overline g}\left(\frac{d_M}{2i}\right)\right) $,
where $\mu_i(M)$ is the $ i$-th eigenvalue of $ M$.
\end{theorem}

Now instead we use $\overline M_k$ to denote the complex space form and $B_{\overline g}(r)$ denotes its geodesic ball.
The K\"ahler version of Theorem \ref{thm: Cheng first thm} is the following result.
\begin{theorem}
Let $k$ be constant.
Suppose $M$ is a K\"ahler manifold and $ p\in M $ such that
$\displaystyle \int_{\mathcal B_g(t, p)} \widehat{\mathrm{Ric}}_{k}^\perp (s_k(\rho) \partial_\rho)dV\ge 0$ and $\displaystyle \int_{\mathcal B_g(t, p)} \widehat{ {\mathrm{H}} }_{2k} (s_{k}(2\rho) \partial_\rho) dV\ge 0$
for all $ 0\le \rho\le r<\mathrm{diam}(\overline M_k)$.
Then
$\lambda_1(\mathcal B_g(r, p)) \le\lambda_1(B_{\overline g} (r))$,
where $\lambda_1$ is the first eigenvalue with respect to the Dirichlet boundary condition.
The equality holds if and only if $\mathcal B_g(r, p)$ is isometric to $ B_{\overline g}(r)$.
\end{theorem}

The proof uses the following
\begin{proposition}
Let $\phi, \psi$ be defined as in Proposition \ref{prop: radial}.
Let $(M, g)$ be a K\"ahler manifold.
Suppose
\begin{align*}
\frac{1}{s_{k_1}(t)^2}\int_{\mathcal B_g'(t, p)}
\widehat{\mathrm{Ric}}_{k_1}^\perp (s_{k_1}(\rho) \partial_\rho)dV+\frac{1}{s_{k_2}(t)^2}\int_{\mathcal B_g'(t, p)} \widehat{ {\mathrm{H}} }_{\frac{k_2}{2}} (s_{k_2}(\rho) \partial_\rho) dV\ge 0
\end{align*}
for all $ 0\le t\le r$. Then
\begin{align*}
\int_{\mathcal B_g(r, p)}\langle \nabla (\psi\circ d_p), \nabla (\phi\circ d_p)\rangle \le -\int_{\mathcal B_g(r, p)} \left(\psi\circ d_p\right)\cdot(\overline \Delta \phi)\circ d_p
\end{align*}
where $\overline \Delta \phi(r):=\phi''+\frac{\overline F'(r)}{\overline F(r)}$ is the Laplacian of $ \phi$ with respect to the metric $ \overline g$ defined in \eqref{eq: model} and $\overline F(t)=s_{k_1}(t)^{2n-2}s_{k_2}(t)$.
\end{proposition}
\begin{proof}
The proof is the same as Proposition \ref{prop: radial} except we replace the last line of \eqref{ineq: radial2} by
\begin{align*}
-\int_{0}^{r}\phi (t) |\phi'(t)|
\left(
\frac{1}{s_{k_1}(t)^2}\int_{\mathcal B_g(t, p)} \widehat{\mathrm{Ric}}_{k_1}^\perp (s_{k_1}(\rho) \partial_\rho)dV
+ \frac{1}{s_{k_2}(t)^2}\int_{\mathcal B_g(t, p)} \widehat{ {\mathrm{H}} }_{\frac{k_2}{2}} (s_{k_2}(\rho) \partial_\rho) dV\right) dt\le 0,
\end{align*}
which follows from Theorem \ref{thm: kahler}.

\end{proof}

\end{document}